\documentclass[10.5pt]{article}
\usepackage{color}
\usepackage{xcolor}
\usepackage{amsfonts}

\usepackage{graphicx}
\usepackage{amsmath,enumerate}
\usepackage{cases}
\usepackage{amsthm}
\usepackage{amssymb}
\usepackage{amsfonts}
\usepackage{latexsym,amssymb,amsfonts}
\usepackage{mathrsfs}
\usepackage{graphicx}
\usepackage{psfrag}
\usepackage{amsmath, amsbsy}
\usepackage{amsopn, amstext}
\usepackage{amsmath,multicol,amsopn}
\usepackage{latexsym,amssymb}
\usepackage{hyperref}
\usepackage{array}
\usepackage{booktabs}
\usepackage{tabularx}
\usepackage{float}
\usepackage{lineno}%  line numbers  %>>> ISSUE: Line 22: garbled comment - should be "% line numbers"
\def\dbF{{\mathbb{F}}}

\def\dbN{{\mathbb{N}}}

\def\dbR{{\mathbb{R}}}

\def\dbV{{\mathbb{V}}}

\def\e{\varepsilon}

\def\f{\varphi}

\def\3n{\negthinspace \negthinspace \negthinspace }
\def\2n{\negthinspace \negthinspace }
\def\1n{\negthinspace }
\def\ns{\noalign{\smallskip} }

\def\ds{\displaystyle}
\def\cC{{\cal C}}
\def\cD{{\cal D}}

\def\cF{{\cal F}}

\def\cJ{{\cal J}}

\def\cU{{\cal U}}
\def\cV{{\cal V}}

\def\mE{{\mathbb{E}}}

\def\ss{\smallskip}

\def\q{\quad}
\def\qq{\qquad}

\def\liminf{\mathop{\underline{\rm lim}}}

\def\wt{\widetilde}
\def\cd{\cdot}

\def\as{\hbox{\rm a.s.{ }}}

\def\({\Big (}
\def\){\Big )}
\def\[{\Big[}
\def\]{\Big]}

\def\={\buildrel \triangle \over =}

\def\ee{\end{equation}}
\def\bea{\begin{eqnarray}}
\def\eea{\end{eqnarray}}
\def\bt{\begin{theorem}}
\def\et{\end{theorem}}
\def\bc{\begin{corollary}}
\def\ec{\end{corollary}}
\def\bl{\begin{lemma}}
\def\el{\end{lemma}}
\def\bp{\begin{proposition}}
\def\ep{\end{proposition}}
\def\br{\begin{remark}}
\def\er{\end{remark}}
\def\ba{\begin{array}}
\def\ea{\end{array}}
\def\bde{\begin{definition}}
\def\ede{\end{definition}}

\newtheorem{lemma}{Lemma}[section]
\newtheorem{remark}{Remark}[section]
\newtheorem{example}{Example}[section]

\newtheorem{theorem}{Theorem}[section]
\newtheorem{corollary}{Corollary}[section]

\newtheorem{definition}{Definition}[section]
\newtheorem{proposition}{Proposition}[section]

\newtheorem{assumption}{Assumption}

\def\punct{}
\newtheoremstyle{dotless}{}{}{\rm}{}{\bf}{\punct}{.5em}{}
\theoremstyle{dotless}

\allowdisplaybreaks[4]
\makeatletter

\@addtoreset{equation}{section}
\makeatother

\title{\bf Stackelberg Stochastic Linear-Quadratic Differential Games: A Closed-Loop Equilibrium Approach
}

\author{
Qi L\"{u}\thanks{School
of Mathematics, Sichuan University, Chengdu
610064, China. (Email: {\tt lu@scu.edu.cn}). This author is supported in part by the NSF of China under grant  12025105.}
~~~~
Bowen Ma\thanks{School of Mathematical Sciences, Chengdu University of Technology, Chengdu 610059, China
(Email: {\tt albertmabowen@gmail.com}).  This author is supported in part by  Sichuan Science and Technology Program under 2026NSFSC0779.}
~~~~
Hanxiao Wang\thanks{School of Mathematical Sciences, Shenzhen University, Shenzhen
518060, China (Email: {\tt hxwang@szu.edu.cn}).  This author is supported in part by the  NSF of China under grant 12522121.}
}

\date{ }

\begin{document}
\maketitle

\vspace{-5mm}

\begin{abstract}
This paper addresses a Stackelberg stochastic linear-quadratic (LQ) differential game under closed-loop information, a problem inherently time-inconsistent. Existing approaches rely on solving two coupled Hamilton-Jacobi-Bellman (HJB) equations derived via time discretization and a limiting argument, whose convergence remains an open problem (see Bensoussan et al. \cite{Bensoussan-Chen-Chutani-Sethi-Siu-Yam-2019}).

We propose an alternative framework based on closed-loop equilibrium strategies. Under closed-loop information, we reformulate the leader's problem as a forward-backward optimal control problem involving a coupled system of forward stochastic differential equations and backward Riccati equations. Due to the presence of controlled Riccati equations, the leader's problem becomes essentially nonlinear. Using a variational method, we characterize the leader's closed-loop equilibrium strategy and derive the associated equilibrium Riccati equation (ERE).

A key conceptual distinction from \cite{Bensoussan-Chen-Chutani-Sethi-Siu-Yam-2019} is that  the follower adopts a globally optimal strategy against any admissible control of the leader, whereas in \cite{Bensoussan-Chen-Chutani-Sethi-Siu-Yam-2019} the follower's strategy is only locally optimal along the leader's specific equilibrium path. This makes the follower's strategy more robust and the leader's commitment more credible. Interestingly, in our LQ setting, the resulting ERE coincides exactly with the coupled HJB system from the literature, showing that the leader's closed-loop equilibrium strategy is equivalent to the so-called feedback Stackelberg solution studied in \cite{Bensoussan-Chen-Chutani-Sethi-Siu-Yam-2019}. Thus, our framework provides not only an alternative derivation but also a rigorous justification of the limiting argument in \cite{Bensoussan-Chen-Chutani-Sethi-Siu-Yam-2019}.

We establish a priori estimates for the ERE under suitable conditions, covering both one-dimensional and high-dimensional cases, thereby ensuring its global well-posedness for any finite horizon. This significantly extends existing results of Ba\c{s}ar and Olsder \cite{Basar-Olsder-1999} and Huang and Shi \cite{Huang-Shi-2024}, which require either a sufficiently short time horizon or control-independent diffusion. An application to an asset management problem with numerical simulations is provided to illustrate the theoretical results.
\end{abstract}

\section{Introduction}

In this paper, we consider a Stackelberg stochastic differential game in the linear-quadratic (LQ) setting. To formulate the game, we introduce the following conventions. Let $W$ be a standard one-dimensional Brownian motion defined on a complete filtered probability space $(\Omega,\mathcal{F},\mathbb{F},\mathbb{P})$, where $\mathbb{F}=\{\mathcal{F}_t\}_{t\geq 0}$ is the augmented natural filtration generated by $W$. Denote by $L_{\mathcal{F}_t}^2(\Omega;\mathbb{R}^n)$ the space of $\mathcal{F}_t$-measurable, square-integrable $\mathbb{R}^n$-valued random variables. Let $\mathbb{S}^n$ be the subspace of $\mathbb{R}^{n\times n}$ consisting of symmetric matrices, and let $\mathbb{S}_+^n$ denote the subset of $\mathbb{S}^n$ comprising positive semidefinite matrices. We use $\mathcal{C}(\kappa_1,\kappa_2,\dots)$ to denote generic constants depending on $\{\kappa_1,\kappa_2,\dots\}$, and simply write $\mathcal{C}$ when no ambiguity arises.

For any given initial pair $(t, \xi)\in [0,T) \times L_{\mathcal{F}_t}^2(\Omega;\mathbb{R}^n)$, consider the following stochastic differential equation (SDE) controlled by two players (Player $1$ uses control $u$ and Player $2$ uses control $v$):
\vspace{-1mm}
\begin{equation}\label{state}
\!\!\!\begin{cases}\ds
dX(s)=\big(A(s)X(s)+ B_1(s)u(s)+B_2(s) v(s) \big)ds \\ \ns\ds\qq \qq+ \big(C(s)X(s)+ D_1(s)u(s)+ D_2(s)v(s) \big)dW(s), & s\in [t,T],\\
\ns\ds X(t)=\xi.
\end{cases}
\end{equation}
Here $A,C \in L^{\infty}(0,T;\mathbb{R}^{n \times n})$,
$B_1,D_1 \in L^{\infty}(0,T;\mathbb{R}^{n \times m_1})$, and $B_2,D_2 \in L^{\infty}(0,T;\mathbb{R}^{n \times m_2})$ for $n,m_1,m_2\in\dbN$, and the controls $u$ and $v$ belong to the strategy sets $\mathcal{U}$ and $\mathcal{V}$, respectively, which will be specified later.
The cost functional of Player $1$ is
\vspace{-2mm}
\begin{equation}\label{cost-leader}
\mathcal{J}_1(t,\xi;u,v)=\frac{1}{2}\mathbb{E}_t \bigg[ \int_{t}^{T} \Big( \langle Q_1(s)X(s),X(s) \rangle + \langle R_1(s) u(s),u(s) \rangle \Big) ds + \langle G_1 X(T),X(T) \rangle \bigg],%\vspace{-1mm}
\end{equation}
\vspace{-1mm}
and the cost functional of Player $2$ is
\vspace{-1mm}
\begin{equation}\label{cost-follower}
\!\!\!\!\!\!\!\!\!\!\!\!\mathcal{J}_2(t,\xi;u,v)= \frac{1}{2}\mathbb{E}_t \bigg[ \int_{t}^{T} \Big( \langle Q_2(s)X(s),X(s) \rangle + \langle R_2(s) v(s),v(s) \rangle \Big) ds + \langle G_2 X(T),X(T) \rangle \bigg],%\vspace{-1mm}
\end{equation}
where $\mathbb{E}_t$ denotes conditional expectation with respect to $\mathcal{F}_t$, $Q_1,Q_2 \in L^{\infty}(0,T;\mathbb{S}^n)$, $R_1 \in L^{\infty}(0,T;\mathbb{S}^{m_1})$, $R_2 \in L^{\infty}(0,T;\mathbb{S}^{m_2})$, and $G_1,G_2 \in \mathbb{S}^n$.

In what follows, we omit the subscript for inner products, as the space is clear from the context, and suppress the time argument whenever no confusion arises.

In the above, Player $1$ (the leader) holds a dominant position by announcing her control $u$ before Player $2$ (the follower). Given the leader's decision, the follower optimally responds with $v^*(u)$. Anticipating this rational reaction, the leader chooses $u^*$ to optimize her own performance along the follower's best-response curve. The pair $(u^*, v^*(u^*))$ constitutes a {\it Stackelberg solution}, summarized as follows:
\vspace{-2mm}
\begin{equation}\label{De-Sta-sol}
\begin{cases}
\mathcal{J}_2(0,x;u,v^*(u)) \leq \mathcal{J}_2(0,x;u,v),\\[4pt]
\mathcal{J}_1(0,x;u^*,v^*(u^*)) \leq \mathcal{J}_1(0,x;u,v^*(u)),
\end{cases}
\qquad \forall (u,v)\in \mathcal{U}\times \mathcal{V}.
\vspace{-1mm}
\end{equation}
This hierarchical solution concept was first introduced by H. von Stackelberg \cite{Stackelberg-1934} for games where some firms dominate others. In such games, the information structure of the strategy sets $\mathcal{U}$ and $\mathcal{V}$ significantly affects the solution. We now describe several types of information structures, present their corresponding Stackelberg solutions, and discuss the resulting time-inconsistency, which motivates the need for alternative solution concepts.

\ss

{\bf Open-loop Stackelberg Solution.} The first one is the so-called {\it (adapted) open-loop information structure}, with control strategy sets defined as follows:
\vspace{-2mm}
\begin{equation}\label{AOL-control}
\begin{cases}
\mathcal{U} = \big\{ u \mid u: \{x\} \times \Omega \times [0,T] \to \mathbb{R}^{m_1} \text{ is an } \mathbb{F}\text{-adapted process} \big\},\\[2pt]
\mathcal{V} = \big\{ v \mid v: \{x\} \times \Omega \times [0,T] \times \mathcal{U} \to \mathbb{R}^{m_2} \text{ is an } \mathbb{F}\text{-adapted process} \big\}.
\end{cases}\vspace{-1mm}
\end{equation}
The associated Stackelberg solution is called the {\it open-loop Stackelberg solution}, first systematically studied by Yong \cite{Yong-2002} for stochastic LQ differential games. Subsequent extensions include asymmetric information \cite{Shi-Wang-Xiong-2016}, jump-diffusion models \cite{Moon-2021}, zero-sum games \cite{Sun-Wang-Wen2021}, multilevel hierarchy \cite{Li-Yu-2018}, and time-delay systems \cite{Xu-Shi-Zhang-2018}.

Under standard positivity assumptions, the maximum principle provides both necessary and sufficient conditions for characterizing the open-loop Stackelberg solution. The resulting optimal Hamiltonian system is a coupled linear forward-backward stochastic differential equation (FBSDE) with four components, whose well-posedness remains largely unresolved. Although a decoupling field can be derived as a Riccati-type ordinary differential equation, its solvability is currently known only in very special cases \cite{Bensoussan-Chen-Sethi-2015,Yong-2002}.

\ss

{\bf Closed-loop Stackelberg Solution.}  The second one is the {\it closed-loop (memoryless) information structure}, under which the admissible control strategy sets are defined as follows: 
\vspace{-2mm}
\begin{equation}\label{CLM-control}
\begin{cases}
\mathcal{U} = \big\{ u \mid u: \{x\} \times \Omega \times [0,T] \times \mathbb{R}^n \to \mathbb{R}^{m_1} \text{ is an } \dbF\text{-adapted process} \big\},\\[6pt]
\mathcal{V} = \big\{ v \mid v: \{x\} \times \Omega \times [0,T] \times \mathbb{R}^n \times \mathcal{U} \to \mathbb{R}^{m_2} \text{ is an } \dbF\text{-adapted process} \big\}.
\end{cases}\vspace{-1mm}
\end{equation}
The key difference from \eqref{AOL-control} is that players now use current state information when choosing controls. This seemingly mild modification substantially complicates the derivation of the Stackelberg solution, now termed the {\it closed-loop Stackelberg solution}.

First, variations in the follower's control $v \in \mathcal{V}$ affect the leader's control $u$, since $u$ depends explicitly on the current state. Thus, when seeking an optimal response, the follower must account for the retroactive effect of her control on the leader's decision. Consequently, the leader faces a nonstandard FBSDE optimal control problem where both $u$ and its state derivative $\partial u/\partial x$ enter the state dynamics.

Second, Bensoussan, Chen, and Sethi \cite{Bensoussan-Chen-Sethi-2015} reformulate this problem as a classical one by restricting the leader's control to the affine form $u(t,x)=u_1(t)+u_2(t)x$, with $u_2$ satisfying a priori bounds. Under this approach, the maximum principle yields a nonlinear FBSDE with four components, whose well-posedness remains largely unresolved. Indeed, closed-loop Stackelberg solutions have long been recognized as intrinsically difficult to analyze; see, e.g., Ba\c{s}ar and Olsder \cite{Basar-Olsder-1999}, Simaan and Cruz \cite{Simaan-Cruz-1973}, and Hern\'{a}ndez et al. \cite{Hernandez-etal-2024}.

\vspace{-2mm}

\subsection{Time-inconsistency}
\vspace{-1mm}

In both open-loop and closed-loop Stackelberg solutions, the leader commits to a strategy at the initial time and dominates the entire game horizon. However, as early as 1973, Simaan and Cruz \cite{Simaan-Cruz-1973} observed in a discrete-time example that the closed-loop Stackelberg solution is time-inconsistent: a solution derived at the initial time may fail to remain a Stackelberg solution at later times. Papavassilopoulos and Cruz \cite{Papavassilopoulos-Cruz-1979} later investigated this phenomenon in continuous-time settings.

To illustrate the issue, we construct Example \ref{example1.1} below. But first, we provide a mathematical definition of time-inconsistency in Stackelberg games.

Let $(u^*, v^*)$ be a Stackelberg solution for the game starting at time $0$ with initial state $x_0$. For any intermediate time $t \in (0, T)$, consider the same game starting at $(t, X^*(t))$, where $X^*$ is the state trajectory induced by $(u^*, v^*)$. Let $(\tilde{u}^*, \tilde{v}^*)$ be a Stackelberg solution for this subgame. The original solution $(u^*, v^*)$ is called time-consistent if
\vspace{-1mm}
$$
\big(u^*(s, \cdot), v^*(s, \cdot, u^*)\big) = \big(\tilde{u}^*(s, \cdot), \tilde{v}^*(s, \cdot, \tilde{u}^*)\big) \quad \text{for almost all } s \in [t, T],\vspace{-1mm}
$$
otherwise it is  time-inconsistent.

Equivalently, from a dynamic programming perspective, time-inconsistency means that the Bellman optimality principle fails: the optimality of a global strategy does not imply the optimality of its restriction to any subinterval.
We now present a concrete example illustrating this phenomenon.%\vspace{-2mm}
\begin{example}\label{example1.1}
Consider the Stackelberg LQ game \eqref{state}--\eqref{cost-follower} with $C=D_1=D_2=0$ and $A,B_1,B_2,Q_1,Q_2$, $R_1,R_2,G_1,G_2$ being constants. We assume that $\min\{R_1,R_2,G_1,G_2,Q_1\}>0$ and $0 \not\in \{B_1,B_2,x_0\}$. By the maximum principle derived in \cite{Bensoussan-Chen-Sethi-2015}, we can show the time-inconsistency under both open-loop and closed-loop information structures as follows:

1. {\bf Open-loop Stackelberg solution.}
From \cite[Theorem 5.1]{Bensoussan-Chen-Sethi-2015}, $(u^*,v^*)$ is an open-loop Stackelberg solution over $[0,T]$ if and only if the following FBSDE admits a solution
\vspace{-1mm}
\begin{equation}\label{optimal-Hamiltonian-open}
\begin{cases}
\dot{x}^* =Ax^* -B_1^2R_1^{-1}p^*_1 -B_2^2R_2^{-1}p^*_2, &\mbox{ in } [0,T],\\
\dot{p}^*_2 =-A p^*_2 -Q_2x^*, &\mbox{ in } [0,T],\\
\dot{y}^* =Ay^* +B_2^2R_2^{-1}p^*_1, &\mbox{ in } [0,T],\\
\dot{p}^*_1 =-Ap^*_1 +Q_2y^*(t)-Q_1x^*, &\mbox{ in } [0,T], \\
x^*(0)=x_0,\qq y^*(0)=0,\\
p^*_1(T)=-G_2y^*(T)+G_1x^*(T),\qq p^*_2(T)=G_2x^*(T),
\end{cases}%\vspace{-1mm}
\end{equation}
with $(u^*,\,v^*)=-(R_1^{-1}B_1p^*_1,\,R_2^{-1}B_2p^*_2)$.
\ss

2. {\bf Closed-loop Stackelberg solution.}
From \cite[Subsection 5.2]{Bensoussan-Chen-Sethi-2015}, we know that the leader will lose nothing if she chooses the affine function $u(t,x)=u_2(t)x+u_1(t)$, and to ensure that the Hamiltonian is finite we need to impose an a priori bound $|u_2|\leq K$.
Then if $\big((u_1^*,u_2^*)^\top,v^*\big)$ is a closed-loop Stackelberg solution, it  satisfies the following necessary condition:
\vspace{-1mm}
\begin{equation}\label{optimal-Hamiltonian-closed1}
u_1^*= u_2^*x^*-R_1^{-1}B_1p_1^*,\qq u_2^*={\rm sgn}(B_1y^*p_2^*)K,\qq v^*=-R_2^{-1}B_2p_2^*,
\end{equation}
\vspace{-3mm}
\begin{equation}\label{optimal-Hamiltonian-closed}
\begin{cases}
\dot{x}^* =(A+B_1u_2^*)x^* +B_1u^*_1-B_2^2R_2^{-1}p^*_2, &\mbox{ in } [0,T],\\
\dot{p}^*_2=-(A+B_1u_2^*) p^*_2 -Q_2x^*,  &\mbox{ in } [0,T],\\
\dot{y}^*=(A+B_1u_2^*)y^* +B_2^2R_2^{-1}p^*_1,  &\mbox{ in } [0,T],\\
\dot{p}^*_1=-(A+B_1u_2^*)p^*_1 +Q_2y^* -Q_1x^* -R_1u_2^* (u_2^* x^* +u_1^*), &\mbox{ in } [0,T], \\
x^*(0)=x_0,\qq y^*(0)=0,\\
p^*_1(T)=-G_2y^*(T)+G_1x^*(T),\qq p^*_2(T)=G_2x^*(T),
\end{cases}
\end{equation}
with\vspace{-1mm}
$$
{\rm sgn}(x)=\begin{cases}
1& \text{if } x>0,\\0 & \text{if } x=0,\\ -1 &\text{if }x<0.
\end{cases}\vspace{-1mm}
$$

For both cases, suppose for contradiction that time-consistency holds. At any $\tilde t \in (0,T)$, restart the game from $x^*(\tilde t)$ and let $(\tilde x^*,\tilde p_2^*,\tilde y^*,\tilde p_1^*)$ with $(\tilde u^*,\tilde v^*)$ (or $((\tilde u_1^*,\tilde u_2^*)^\top,\tilde v^*)$) be the corresponding solution over $[\tilde t,T]$. Time-consistency then implies $u^* = \tilde u^*$ and $x^* = \tilde x^*$ on $[\tilde t,T]$ (or $u_1^* = \tilde u_1^*$, $u_2^* = \tilde u_2^*$, $x^* = \tilde x^*$), hence $p_1^* = \tilde p_1^*$ on $[\tilde t,T]$ and consequently $y^*(T) = \tilde y^*(T)$. Since $y^*$ and $\tilde y^*$ satisfy the same equation with the same terminal condition, we obtain $\tilde y^* = y^*$ on $[\tilde t,T]$, yielding $y^*(\tilde t) = 0$. By the arbitrariness of $\tilde t \in (0,T)$, $y^* \equiv 0$ on $[0,T]$. Substituting this into the equation for $y^*$ gives $p_1^* \equiv 0$ on $[0,T]$, and thus $u^* \equiv 0$ (or $u_1^* = u_2^* \equiv 0$) on $[0,T]$. Finally, from the equation for $p_1^*$, we obtain $x^* \equiv 0$ on $[0,T]$, contradicting $x_0 \neq 0$.
\end{example}

To address the time-inconsistency under the open-loop information structure, Moon and Yang \cite{Moon-Yang-2020} derived the corresponding equilibrium Riccati equations via a variational approach. However, the solvability of these equations remains unclear due to the lack of symmetry and the intrinsic complexity of the system.

For time-inconsistency under the closed-loop information structure, Simaan and Cruz \cite{Simaan-Cruz-1973} introduced the notion of a {\it feedback Stackelberg solution} for discrete-time systems. Motivated by dynamic programming (DPP), this solution concept is defined recursively by solving a sequence of static Stackelberg games at each stage in a backward manner. Later, Ba\c{s}ar and Haurie \cite{Basar-Haurie-1984} extended this concept to the continuous-time setting via time discretization and a limiting argument, deriving a system of coupled Hamilton-Jacobi-Bellman (HJB) equations involving a static Stackelberg game at the Hamiltonian level.   Although recognized as essential to the methodology, the justification of the limiting procedure was left as a challenging open problem \cite{Bensoussan-Chen-Chutani-Sethi-Siu-Yam-2019,Bensoussan-Chen-Sethi-2014}. Instead, they bypassed the limiting procedure by introducing the notion of a \textbf{feedback Stackelberg equilibrium}, which is intrinsically time-consistent by definition: if for any $(t,x)\in [0,T]\times \mathbb{R}^n$,
\vspace{-1mm}
\begin{equation}\label{De-feed-Sta}
\begin{cases}
\mathcal{J}_2(t,x;u^*,v^*(u^*)) \leq \mathcal{J}_2(t,x;u^*,v(u^*)),\\[4pt]
\mathcal{J}_1(t,x;u^*,v^*(u^*)) \leq \mathcal{J}_1(t,x;u,v^*(u)),
\end{cases}\qquad \forall (u,v)\in \mathcal{U}\times \mathcal{V},\vspace{-1mm}
\end{equation}
where
\vspace{-1mm}
\begin{equation}\label{feed-Sta-control}
\begin{cases}
\mathcal{U} = \big\{ u \mid u: [0,T]\times \mathbb{R}^n \to \mathbb{R}^{m_1} \text{ and } u(s,x) \text{ is Lipschitz continuous in } x \big\},\\[2pt]
\mathcal{V} = \big\{ v \mid v: [0,T] \times \mathbb{R}^n \times \mathbb{R}^{m_1} \to \mathbb{R}^{m_2} \text{ and } v(s,x,u) \text{ is Lipschitz continuous in } (x,u) \big\},
\end{cases}\vspace{-1mm}
\end{equation}
then the pair of control strategies $(u^*,v^*)\in \mathcal{U}\times \mathcal{V}$ is called a \emph{feedback Stackelberg equilibrium}. Nevertheless, the only available approach in the literature for obtaining the feedback Stackelberg equilibrium \eqref{De-feed-Sta} still relies on solving the coupled HJB system mentioned above.

Along this line, Bensoussan, Chen, and Sethi \cite{Bensoussan-Chen-Sethi-2014} studied an infinite-horizon Stackelberg game; Bensoussan et al. \cite{Bensoussan-Chen-Chutani-Sethi-Siu-Yam-2019} investigated a mixed leadership game; and Huang and Shi \cite{Huang-Shi-2024} considered a finite-horizon LQ game where the control enters the diffusion term. However, even in the LQ case, where the HJB system reduces to coupled Riccati equations, solvability is only known when the diffusion is control-independent and $T$ is sufficiently small \cite[Chapter 7.6.2]{Basar-Olsder-1999} \cite[Theorem 2]{Huang-Shi-2024}. When the control enters the diffusion term, the problem becomes much harder and remains largely open \cite[Remark 3]{Huang-Shi-2024}.

Moreover, Definition \eqref{De-feed-Sta} reveals that the follower's strategy is only locally optimal. This is unrealistic, as the leader cannot force the follower to accept a local optimum in place of a global one. A locally optimal strategy also lacks robustness against the leader's choice. Therefore, in contrast to the existing literature on feedback Stackelberg equilibrium, we insist that the follower adopt a globally optimal strategy, as in open-loop or closed-loop Stackelberg solutions. This is the fundamental starting point of our paper.

\vspace{-1mm}

\subsection{Time-inconsistent control approach and closed-loop equilibrium solution}

Based on the above considerations, this paper aims to introduce a new approach to addressing the time-inconsistency of Stackelberg games under a closed-loop information structure. The main idea is motivated by time-inconsistent control theory \cite{Bjork-2017,Ekeland2008,Hu-2012,Hu-2017,Yong-2014,Yong-2017}, building in particular on our previous work on time-inconsistent optimal control for forward-backward stochastic control problems \cite{Lu-2024,Lu-Ma-Wang-2025,Wang-Yong-Zhou-2024}.

The main procedure of our approach can be divided into four steps:

\vspace{-1mm}

\textbf{Step 1.} We assume the leader takes the linear feedback control $u = \Theta_1 X$ with $\Theta_1 \in L^\infty(0,T; \mathbb{R}^{m_1 \times n})$. For each initial pair $(t,\xi) \in [0,T) \times L_{\mathcal{F}_t}^2(\Omega; \mathbb{R}^n)$, the follower faces a standard stochastic LQ optimal control problem, which admits a unique (time-consistent) closed-loop optimal strategy $\Theta_2^*$ that depends on $\Theta_1$. Observe that the follower's strategy $\Theta_2^*$ is optimal for any given $\Theta_1$; hence, it is globally optimal.

\ss

\textbf{Step 2.} After anticipating the follower's rational reaction, we reformulate the leader's problem as the following optimal control problem for a coupled system consisting of a forward SDE and a backward Riccati equation, together with a quadratic cost functional:
\vspace{-2mm}
\begin{eqnarray}\label{SDE-Riccati}
\begin{cases}
\!\!\begin{cases}
dX=[(A+ B_1 \Theta_1 ) X+ B_2 \Theta_2^*(\Theta_1)X ]ds+ [(C+D_1 \Theta_1 )X + D_2 \Theta_2^*(\Theta_1)X ]dW(s),\q \mbox{in }[t,T],\\\ns\ds
\dot{P}_2=- ( A+B_1\Theta_1)^\top P_2 - P_2(A+B_1\Theta_1)- (C+D_1\Theta_1)^\top P_2 (C+D_1\Theta_1) - Q_2\\\ns\ds
\hspace{2.5em}+[P_2B_2\!+\! (C\!+\!D_1\Theta_1)^\top P_2 D_2 ][R_2\!+\! D_2^\top P_2 D_2]^{-1} [B_2^\top P_2\!+\! D_2^\top P_2 (C\!+\!D_1\Theta_1) ],\q \mbox{in }[t,T],\\
X(t)=\xi,\q P_2(T)=G_2.
\end{cases}\\\ns\ds
\cJ_1(t,\xi; \Theta_1 X,\Theta_2^*(\Theta_1)X)=
\frac{1}{2}\mE_t \Big[ \int_{t}^{T}  \langle  ( Q_1+ \Theta_1^\top R_1 \Theta_1  )X,X \rangle ds +  \langle G_1X(T),X(T) \rangle  \Big],
\end{cases}%\vspace{-1mm}
\end{eqnarray}
where
\begin{equation}\label{In-Theat2}
\Theta_2^*(\Theta_1) = - \big[R_2 + D_2^\top P_2 D_2\big]^{-1} \big[ B_2^\top P_2 + D_2^\top P_2 (C + D_1\Theta_1) \big].%\vspace{-1mm}
\end{equation}
The leader aims to minimize the cost functional $\mathcal{J}_1(t,\xi; \Theta_1 X, \Theta_2^*(\Theta_1) X)$ by choosing $\Theta_1$ in the space $L^\infty(0,T; \mathbb{R}^{m_1 \times n})$.

\ss

\textbf{Step 3.} The leader faces an optimal control problem for a nonlinear FBSDE, which is time-inconsistent. To address the leader's time-inconsistency, we introduce the notion of a (linear) closed-loop equilibrium strategy (see Definition \ref{De-closedloop}). By a variational approach together with decoupling techniques, the following equilibrium Riccati equation (ERE) is derived:
\vspace{-1mm}
\begin{align}\label{Intro-ERE}
\!\!\begin{cases}
\dot{P}_2 = - (A + B_1 \bar{\Theta}_1)^\top P_2 - P_2(A + B_1 \bar{\Theta}_1) - (C + D_1 \bar{\Theta}_1)^\top P_2 (C + D_1 \bar{\Theta}_1) - Q_2 \\[2pt]
\qquad + \big[P_2 B_2 + (C + D_1 \bar{\Theta}_1)^\top P_2 D_2 \big] \big[R_2 + D_2^\top P_2 D_2\big]^{-1} \big[B_2^\top P_2 + D_2^\top P_2 (C + D_1\bar{\Theta}_1) \big], \quad \text{in } [0,T],\\[2pt]
\dot{P}_1 = - \big[ \mathbf{A}(\bar{\Theta}_1,P_2)^\top P_1 + P_1 \mathbf{A}(\bar{\Theta}_1,P_2) + \mathbf{C}(\bar{\Theta}_1,P_2)^\top P_1 \mathbf{C}(\bar{\Theta}_1,P_2) + Q_1 + \bar{\Theta}_1^\top R_1 \bar{\Theta}_1 \big], \quad \text{in } [0,T],\\[2pt]
P_2(T) = G_2, \quad P_1(T) = G_1,
\end{cases}\vspace{-1mm}
\end{align}
with
\vspace{-1mm}
\begin{equation}\label{Intro-bar-Theta1}
\!\!\!\bar{\Theta}_1 = \!- \big[R_1 + \mathbf{D}(P_2)^\top\! P_1 \mathbf{D}(P_2)\big]^{-1} \!\big\{ \mathbf{B}(P_2)^\top\! P_1 + \mathbf{D}(P_2)^\top\! P_1 \big[ C - D_2 (R_2 + D_2^{\top} P_2 D_2)^{-1} (B_2^\top P_2 + D_2^\top P_2 C) \big] \big\},\vspace{-1mm}
\end{equation}
which provides an equivalent characterization for the leader's (linear) closed-loop equilibrium strategy (see Theorem \ref{th1}). Combining this closed-loop equilibrium strategy $\bar{\Theta}_1$ with the follower's rational reaction mapping $\Theta_2^*(\bar{\Theta}_1)$ yields a time-consistent solution to the global Stackelberg game under the closed-loop information structure.

\ss

\textbf{Step 4.} By establishing appropriate a priori estimates, we show that the ERE admits a unique solution (Theorems \ref{th3} and \ref{th4}) over an arbitrary time horizon $[0,T]$, provided that any of the following conditions holds:

(i) $D_2 = 0$, with $\operatorname{Rank} (B_2(t)) = n$, and $B_2(\cdot)$ and $R_2(\cdot)$ are continuously differentiable on $[0,T]$;

(ii) $m_1 = m_2 = n = 1$ with $|D_2| > \delta > 0$;

(iii) $m_1 = m_2 = n = 1$ with $D_2 = 0$.

\noindent
Intuitive explanations of the above technical assumptions are also provided. Lastly, we give an application of the obtained results to an asset management problem (Section \ref{sec5}).

\ss

An interesting finding is that the closed-loop equilibrium strategy pair $(\bar{\Theta}_1, \Theta_2^*(\bar{\Theta}_1))$ constitutes a feedback Stackelberg equilibrium in the sense of Definition~\eqref{De-feed-Sta} (see Theorem \ref{th2}). Moreover, the resulting ERE coincides with the coupled Riccati equation derived from the HJB system of a static Stackelberg game at the Hamiltonian level. A detailed comparison between the closed-loop equilibrium strategy \eqref{De-closedloop} and the feedback Stackelberg equilibrium \eqref{De-feed-Sta} is given in Table \ref{Tab1} and Remark \ref{remark3.3}. The key distinction is that in our approach, the follower adopts a globally optimal strategy, whereas in \cite{Bensoussan-Chen-Chutani-Sethi-Siu-Yam-2019} the follower's strategy is only locally optimal, implying stronger robustness for the follower in our framework.

The main technical challenges are twofold. First, because the leader's   system involves a Riccati equation, the leader's problem becomes a nonlinear forward-backward stochastic control problem. Thus, although the original Stackelberg game \eqref{state}--\eqref{cost-leader}--\eqref{cost-follower} is LQ, the leader's problem falls outside the LQ framework, distinguishing our work from existing studies \cite{Bensoussan-Chen-Sethi-2015,Li-Yu-2018,Moon-2021,Shi-Wang-Xiong-2016,Sun-Wang-Wen2021,Xu-Shi-Zhang-2018,Yong-2002}, where the leader's problem is a linear FBSDE optimal control problem. A key point is that we always take the leader's control in state-feedback form. Second, due to singularity and nonlinearity, solvability of the Riccati equation \eqref{Intro-ERE} is largely open \cite{Basar-Olsder-1999,Bensoussan-Chen-Chutani-Sethi-Siu-Yam-2019,Bensoussan-Chen-Sethi-2014,Huang-Shi-2024}. Existing results \cite{Basar-Olsder-1999,Huang-Shi-2024} typically require $T$ to be sufficiently small and the diffusion to be control-independent, allowing a standard fixed-point argument. In contrast, we establish global well-posedness of \eqref{Intro-ERE} (Theorems \ref{th3} and \ref{th4}) under the assumptions in Step 4, significantly improving existing results in two aspects: (1) the short-time constraint is removed; (2) the diffusion can be controlled. The key lies in establishing appropriate a priori estimates.

The rest of this paper is organized as follows. Section~\ref{se2} reformulates the leader's problem under closed-loop information and collects preliminary results. Section~\ref{se3} is divided into two subsections: Subsection~\ref{se3.1} provides an equivalent characterization of the leader's closed-loop equilibrium strategy and derives the corresponding ERE, while Subsection~\ref{se3.2} shows that this strategy together with the follower's rational reaction forms a feedback Stackelberg equilibrium. Section~\ref{se4} establishes global well-posedness of the ERE under suitable conditions. An application is given in Section~\ref{sec5}, and all proofs are collected in Section~\ref{Appendix}.

\vspace{-2mm}

\subsection{Literature review on time-inconsistent control problems}

The earliest mathematical study of time-inconsistent optimal control is due to Strotz \cite{Strotz-1955}, followed by later works \cite{Bjork-2017,Ekeland-Lazrak-2010,Ekeland2008,Hu-2012,Hu-2017,Yong-2014,Yong-2017}. For forward-backward controlled systems, Wei, Yong, and Yu \cite{Wei-Yong-Yu-2017} studied a recursive problem; Moon and Yang \cite{Moon-Yang-2020} considered an LQ mean-field leader-follower Stackelberg game with open-loop information; Wang, Yong, and Zhou \cite{Wang-Yong-Zhou-2024} investigated a forward-backward system with a recursive cost given by a backward stochastic Volterra integral equation; L\"u and Ma \cite{Lu-2024} characterized closed-loop equilibrium strategies for a time-inconsistent forward-backward LQ problem via a nonsymmetric ERE; and L\"u, Ma, and Wang \cite{Lu-Ma-Wang-2025} established a priori estimates for that ERE and proved its well-posedness in the high-dimensional case. However, none of these works can address the leader's problem \eqref{SDE-Riccati}--\eqref{In-Theat2} in Stackelberg stochastic LQ games. This paper aims to resolve this issue.

\vspace{-2mm}
\section{Preliminary}\label{se2}

In Subsection \ref{se-2.1}, we reformulate the leader's control problem under a closed-loop information structure in detail. Subsection \ref{se-2.2} collects some preliminary results that will be used later.

\vspace{-1mm}
\subsection{Reformulation of the leader's problem}\label{se-2.1}

We impose the following assumptions on the coefficients and weighting matrices:
\begin{assumption}\label{H2}
The weighting matrices satisfy
\vspace{-1mm}
$$
Q_1 \geq 0, \quad Q_2 \geq 0, \q G_1\geq 0,\q G_2\geq 0, \quad \text{and} \quad R_1 \geq \delta I, \quad R_2 \geq \delta I,\vspace{-1mm}
$$
for some $\delta > 0$.
\end{assumption}
For each initial pair $(t,\xi) \in [0,T) \times L_{\mathcal{F}_t}^2(\Omega;\mathbb{R}^n)$, we posit that Player 1 (the leader) adopts the following linear feedback control
\begin{equation}\label{se2.1-eq1}
u = \Theta_1 X \quad \text{with} \quad \Theta_1 \in L^\infty(0,T;\mathbb{R}^{m_1 \times n}),
\end{equation}
which is a particular class of control within the closed-loop information structure.
\begin{remark}
One may consider a more general form of the control ansatz $u = \Theta_1 X + \varphi$. Nevertheless, as will be shown later in Theorem \ref{th1}, the particular form \eqref{se2.1-eq1} is sufficient to address the time-inconsistency inherent in Example \ref{example1.1}.
\end{remark}
With \eqref{se2.1-eq1}, the control system becomes
\vspace{-1mm}
\begin{equation}\label{se2.1-state}
\begin{cases}
dX  = \big[(A + B_1 \Theta_1) X  + B_2 v \big] ds + \big[(C + D_1 \Theta_1) X  + D_2 v \big] dW(s), & \mbox{ in } [t,T],\\
X(t) = \xi,
\end{cases}
\end{equation}
and the cost functional for Player 2 (the follower) is given by
\vspace{-1mm}
\begin{equation}\label{se2.1-follower-cost}
\mathcal{J}_2(t,\xi;\Theta_1 X, v) = \frac{1}{2} \mathbb{E}_t \bigg[ \int_t^T \Big( \langle Q_2 X, X \rangle + \langle R_2 v, v \rangle \Big) ds + \langle G_2 X(T), X(T) \rangle \bigg].\vspace{-1mm}
\end{equation}

By the standard theory of stochastic LQ optimal control \cite{Yong-Zhou1999}, the follower's problem is summarized as follows.
\begin{proposition}\label{pro2.1}
Let Assumption \ref{H2} hold and let $\Theta_1 \in L^\infty(0,T;\mathbb{R}^{m_1 \times n})$ be fixed. Then the follower's problem \eqref{se2.1-state}--\eqref{se2.1-follower-cost} admits a unique optimal control
$v^* = \Theta_2^*(\Theta_1) X$,
where
\vspace{-1mm}
\begin{equation}\label{Th2*}
\Theta_2^*(\Theta_1) = - \big[R_2 + D_2^{\top} P_2 D_2 \big]^{-1} \big[ B_2^\top P_2 + D_2^\top P_2 (C + D_1\Theta_1) \big],\vspace{-1mm}
\end{equation}
$X$ is the unique solution of the closed-loop system with the control $v^* = \Theta_2^*(\Theta_1) X$, and $P_2$ (depending on $\Theta_1$) is the unique solution of the following Riccati equation:
\vspace{-1mm}
\begin{equation}\label{P_2}
\begin{cases}
\dot{P}_2 + (A + B_1\Theta_1)^\top P_2 + P_2(A + B_1\Theta_1) + (C + D_1\Theta_1)^\top P_2 (C + D_1\Theta_1) + Q_2 \\[4pt]
\quad - \big[P_2 B_2 + (C + D_1\Theta_1)^\top P_2 D_2 \big] \big[R_2 + D_2^{\top} P_2 D_2 \big]^{-1} \big[B_2^\top P_2 + D_2^\top P_2 (C + D_1\Theta_1)\big] = 0, & \mbox{ in } [0,T],\\[2pt]
P_2(T) = G_2.
\end{cases}\vspace{-1mm}
\end{equation}
Moreover,  
\vspace{-1mm}
\begin{equation*}\label{value-2}
\mathcal{J}_2\big(t,\xi;\Theta_1 X, \Theta_2^*(\Theta_1) X\big) = \frac{1}{2} \langle P_2(t)\xi, \xi \rangle.\vspace{-1mm}
\end{equation*}
\end{proposition}
Thus, at time $t$, the leader's problem reduces to an optimal control problem for the following coupled system (forward SDE + backward Riccati equation) with a quadratic cost: 
\vspace{-1mm}
\begin{equation}\label{ConPro-leader}
\begin{cases}\2n
\begin{cases}
dX=[(A+ B_1 \Theta_1 ) X+ B_2 \Theta_2^*(\Theta_1)X ]ds+ [(C+D_1 \Theta_1 )X + D_2 \Theta_2^*(\Theta_1)X ]dW(s),\q \mbox{in } [t,T],\\\ns\ds
\dot{P}_2=- ( A+B_1\Theta_1)^\top P_2 - P_2(A+B_1\Theta_1)- (C+D_1\Theta_1)^\top P_2 (C+D_1\Theta_1) - Q_2\\\ns\ds
\hspace{2.5em}+[P_2B_2\!+\! (C\!+\!D_1\Theta_1)^\top P_2 D_2 ][R_2\!+\! D_2^\top P_2 D_2]^{-1} [B_2^\top P_2\!+\! D_2^\top P_2 (C\!+\!D_1\Theta_1) ], \q \mbox{in } [t,T],\\
X(t)=\xi,\q P_2(T)=G_2.
\end{cases}\\\ns\ds
\cJ_1(t,\xi; \Theta_1 X,\Theta_2^*(\Theta_1)X)=
\frac{1}{2}\mE_t \Big[ \int_{t}^{T}  \langle  ( Q_1+ \Theta_1^\top R_1 \Theta_1  )X,X \rangle ds +  \langle G_1X(T),X(T) \rangle  \Big].
\end{cases}\vspace{-1mm}
\end{equation}
\begin{remark}
Note that once the leader fixes the linear feedback control $u = \Theta_1 X$, the follower's problem becomes a standard (and hence time-consistent) stochastic LQ problem. Thus, only the leader's problem may suffer from time-inconsistency, which we address using ideas from time-inconsistent forward-backward LQ control \cite{Lu-2024,Lu-Ma-Wang-2025}.
\end{remark}

\subsection{Preliminary results}\label{se-2.2}

We first introduce some necessary definitions and results from matrix analysis.
\begin{definition}
For any $A = (a_{ij}) \in \mathbb{R}^{n \times m}$, the \textit{vec-operator} stacks the columns of $A$ into a vector:
\vspace{-1mm}
\begin{equation}
\operatorname{vec}(A) = (a_{11}, \dots, a_{n1}, a_{12}, \dots, a_{n2}, \dots, a_{1m}, \dots, a_{nm})^\top.
\end{equation}
\end{definition}
\begin{definition}
For any $A \in \mathbb{R}^{n \times m}$ and $B \in \mathbb{R}^{p \times q}$, the \textit{Kronecker product} of $A$ and $B$ is defined by
\vspace{-1mm}
\begin{equation}
A \otimes B = \begin{pmatrix}
a_{11}B & \cdots & a_{1m}B \\
\vdots & \ddots & \vdots \\
a_{n1}B & \cdots & a_{nm} B
\end{pmatrix}.\vspace{-1mm}
\end{equation}
\end{definition}
\begin{lemma}\cite[Chapter 4.3]{Horn-Johnson-1991}\label{lm2.1}
For $A, B, C \in \mathbb{R}^{n \times n}$, the following matrix equations are equivalent:
\vspace{-1mm}
\begin{equation}
AX + XB = C \quad \Longleftrightarrow \quad (I_n \otimes A + B^\top \otimes I_n) \operatorname{vec}(X) = \operatorname{vec}(C).
\end{equation}
\end{lemma}
\begin{lemma}\cite[Chapter 4.2, Chapter 4.4]{Horn-Johnson-1991}\label{lm2.2}
Let $A, B \in \mathbb{R}^{n \times n}$, and let $\sigma(A) = \{\lambda_i\}_{i=1}^n$ and $\sigma(B) = \{\mu_j\}_{j=1}^n$ be the sets of eigenvalues of $A$ and $B$, respectively. Then the following properties hold:
\begin{itemize}
\item[(i)] $(A \otimes B)^\top = A^\top \otimes B^\top$;
\item[(ii)] The eigenvalues of $I_n \otimes A + B \otimes I_n$ are given by $\{ \lambda_i + \mu_j \mid 1 \leq i \leq n, \; 1 \leq j \leq n \}$.
\end{itemize}
\end{lemma}
We recall existence and uniqueness results for ordinary differential equations with discontinuous right-hand side. To this end, we introduce the \textit{Carathéodory condition} in a domain $\mathcal{D}$ of the $(t,x)$-space:
\begin{equation*}
\text{Carath\'eodory condition: } \begin{cases}
\text{(i) the function $f(t,x)$ is defined and continuous in $x$ for almost all $t$;}\\[2pt]
\text{(ii) the function $f(t,x)$ is measurable in $t$ for each $x$;}\\[2pt]
\text{(iii) $|f(t,x)| \leq m(t)$, where $m(t)$ is integrable.}
\end{cases}
\end{equation*}

With this concept, we summarize the following results.

\begin{lemma}\cite[Chapter 1.1]{Filippov-1988}\label{lm2.3}
Let $\mathcal{D} = [0,T] \times \{ x \in \mathbb{R}^n \mid |x - x_T| \leq b \}$, and suppose
\begin{enumerate}
\item[(i)] the function $f$ satisfies the Carathéodory condition on $\mathcal{D}$;
\item[(ii)] $|f(t,x) - f(t,y)| \leq l(t) |x - y|$ for any $(t,x)$ and $(t,y)$ in $\mathcal{D}$, where $l$ is an integrable function on $[0,T]$.
\end{enumerate}
Then for the ordinary differential equation
\vspace{-1mm}
\begin{equation} \label{bode}
\dot{x}(t) = f(t, x(t)), \quad x(T) = x_T,\vspace{-1mm}
\end{equation}
we have
\begin{enumerate}
\item[(1)] there exists $\delta > 0$ such that \eqref{bode} admits a unique solution over $[T-\delta, T]$;
\item[(2)] the solution can be uniquely continued up to the boundary $\partial \mathcal{D}$ of the domain $\mathcal{D}$.
\end{enumerate}
\end{lemma}

We recall an implicit mapping theorem for matrix functions.
\begin{lemma}\cite[Chapter VI, pp. 417]{Edwards-1973}\label{lm2.4}
Let $A, B \in \mathbb{R}^{n \times n}$, and let $F(X,Y): \mathbb{R}^{n \times n} \times \mathbb{R}^{n \times n} \to \mathbb{R}^{n \times n}$ be continuously differentiable. Denote by $D_X F(A,B): \mathbb{R}^{n \times n} \to \mathbb{R}^{n \times n}$ the partial Fr\'echet derivative of $F$ with respect to $X$ at $(A,B)$, given by
\vspace{-1mm}
\begin{equation*}
D_X F(A,B)[H] := \lim_{\varepsilon \to 0} \frac{F(A + \varepsilon H, B) - F(A,B)}{\varepsilon}.\vspace{-1mm}
\end{equation*}
Suppose $F(A,B) = 0$ and $D_X F(A,B)$ is an isomorphism. Then there exist a neighborhood $U$ of $B$ in $\mathbb{R}^{n \times n}$, a neighborhood $W$ of $(A,B)$ in $\mathbb{R}^{n \times n} \times \mathbb{R}^{n \times n}$, and a continuously differentiable mapping $\varphi: U \to \mathbb{R}^{n \times n}$ such that the following holds: if $(X,Y) \in W$ and $Y \in U$, then $F(X,Y) = 0$ if and only if $X = \varphi(Y)$.
\end{lemma}

\vspace{-1mm}

\section{Closed-loop equilibrium strategy of the leader}\label{se3}

In this section, we first introduce the leader's closed-loop equilibrium strategy (Definition \ref{De-closedloop}) and characterize it via an equilibrium Riccati equation (ERE). We then show that this strategy, together with the follower's rational response, forms a feedback Stackelberg equilibrium \eqref{De-feed-Sta}. A comparison of the two solution concepts is given in Table \ref{Tab1}.

\vspace{-1mm}

\subsection{Equilibrium Riccati equation for Problem \eqref{ConPro-leader}}\label{se3.1}

We begin with the definition of a closed-loop equilibrium strategy for the leader's problem \eqref{ConPro-leader}. This definition captures the idea that the leader should have no incentive to deviate from her strategy at any time, thereby resolving time-inconsistency.

%\vspace{-1mm}
%
\begin{definition}\label{De-closedloop}
A matrix-valued function $\bar{\Theta}_1 \in L^\infty(0,T;\dbR^{m_1\times n})$ is called a closed-loop equilibrium strategy for Problem \eqref{ConPro-leader}
if  for any $ (t,\xi,\dbV)  \in [0,T)\times L_{\cF_t}^2(\Omega;\dbR^n)\times L^\infty(0,T;\dbR^{m_1\times n}),$
the following variational inequality holds:
\vspace{-1mm}
\begin{equation}
\liminf_{\e\to 0+} \frac{\cJ_1\big(t,\xi;\bar{\Theta}^\e_1 X^\e,\Theta_2^*(\bar{\Theta}^\e_1) X^\e\big)-\cJ_1\big(t,\xi;\bar{\Theta}_1 \bar{X},\Theta_2^*(\bar{\Theta}_1) \bar{X} \big)}{\e}\geq 0,  \q \as
\end{equation}
where $\bar{\Theta}_1^\e(s)=\bar{\Theta}_1(s)+1_{[t,t+\e]}\dbV(s)$,
and $X^\e(\cd)$ denotes the solution of the state equation in \eqref{ConPro-leader} corresponding to the perturbed control pair $\big( \bar{\Theta}^\e_1 X^\e,\Theta_2^*(\bar{\Theta}^\e_1) X^\e\big)$.
\end{definition}
\begin{remark}
Definition \ref{De-closedloop} implies that once the leader adopts $\bar{\Theta}_1$, no unilateral deviation at any time $t$ can improve her outcome. Thus, the leader has no regret after committing to this strategy, offering a rational resolution to time-inconsistency. This is analogous to the notion of ``consistent planning" in time-inconsistent control \cite{Bjork-2017,Strotz-1955}.
\end{remark}
%
%\vspace{-3mm}
%
\begin{remark}
The specific perturbation form $\bar{\Theta}_1^\epsilon(s) = \bar{\Theta}_1(s) + \mathbf{1}_{[t,t+\epsilon]}(s)\mathbb{V}(s)$ is chosen for the following reasons:

1.  Richness: As $\epsilon \to 0$, this class of perturbations can generate  any  element in $L_{\mathcal{F}_t}^2(\Omega;\mathbb{R}^n)$ at time $t$, making it sufficiently rich to capture the essential feature of the closed-loop equilibrium strategy.

2. Locality: The perturbation is active only on the infinitesimal interval $[t, t+\epsilon]$, which allows us to isolate the marginal effect of changing the strategy at time $t$.

3.  Analytical tractability: This form leads to a clean variational inequality that can be characterized via an ERE, as shown in Theorem \ref{th1}.

Although different perturbation forms appear in the literature (see, e.g., \cite{Dou-2020,Lu-Ma-Wang-2025,Wang-Yong-Zhou-2024,Yong-2017}), they lead to the same ERE and hence the same equilibrium strategy.
\end{remark}
%
%\vspace{-3mm}
%
\begin{remark}
We focus on linear equilibrium strategies, as this paper concerns sufficiency--i.e., the existence of equilibrium strategies. Whether nonlinear equilibria also exist remains open.
\end{remark}
%
%\vspace{-1mm}
To state the main characterization, we introduce the following shorthand notations:
\vspace{-1mm}
\begin{equation}\label{nota}
\begin{cases}
\ds\mathbf{B}(P_2):= B_1 -B_2(R_2+ D_2^\top P_2 D_2)^{-1} D_2^\top P_2 D_1,\\
\ns\ds \mathbf{D}(P_2):= D_1 -D_2(R_2+ D_2^\top P_2 D_2)^{-1} D_2^\top P_2 D_1,\\
\ns\ds \mathbf{A}(\Theta_1,P_2) =A- B_2\big( R_2+ D_2^{\top} P_2 D_2 \big)^{-1} \big(  B_2^\top P_2+ D_2^\top P_2 C \big) + \mathbf{B}(P_2)\Theta_1,\\
\ns\ds\mathbf{C}(\Theta_1,P_2) =  C  -D_2 \big( R_2+ D_2^{\top} P_2 D_2 \big)^{-1} \big( B_2^\top P_2+ D_2^\top P_2 C \big)+ \mathbf{D}(P_2)\Theta_1.
\end{cases}\vspace{-1mm}
\end{equation}
where $P_2$ depends on $\Theta_1$ through Riccati equation \eqref{P_2}.

Moreover, we introduce the following equation derived via decoupling (see Subsection \ref{Appendix-pr-th1}):
\vspace{-1mm}
\begin{equation}\label{P_1}
\begin{cases}
\dot{P}_1+ \mathbf{A}(\Theta_1,P_2)^\top P_1+ P_1 \mathbf{A}(\Theta_1,P_2) + \mathbf{C}(\Theta_1,P_2)^\top P_1 \mathbf{C}(\Theta_1,P_2) + Q_1+ \Theta_1^\top R_1 \Theta_1 =0,\q \mbox{in } [0,T],\\
P_1(T)=G_1.
\end{cases}\vspace{-1mm}
\end{equation}
Now we can give the equivalent characterization for the closed-loop equilibrium strategy in Definition \ref{De-closedloop} and derive the corresponding ERE. 
\begin{theorem}\label{th1}
Let Assumption \ref{H2} hold. Then a control strategy $\bar{\Theta}_1\in L^\infty(0,T;\dbR^{m_1\times n})$ is a closed-loop equilibrium strategy defined in Definition \ref{De-closedloop} if and only if
\vspace{-1mm}
\begin{equation}\label{th1-eq1}
R_1 \bar{\Theta}_1+   \mathbf{B}(P_2)^\top P_1+ \mathbf{D}(P_2)^\top P_1 \mathbf{C} (\bar{\Theta}_1,P_2)=0.\vspace{-1mm}
\end{equation}
\end{theorem}
By Assumption \ref{H2} and  standard properties of Lyapunov differential equations  \cite[Lemma 7.3]{Yong-Zhou1999}, we have $\mathbf{D}(P_2)^\top P_1 \mathbf{D}(P_2)+ R_1\geq \delta I$.
This together with \eqref{nota} and \eqref{th1-eq1} implies
\vspace{-1mm}
\begin{align}\label{bar-Theta1}
\!\!\!\!\bar{\Theta}_1 \!=\!-\big[R_1\!+\! \mathbf{D}(P_2)^\top\! P_1 \mathbf{D}(P_2)  \big]^{-1}\! \big\{  \mathbf{B}(P_2)^\top\!  P_1\!+\! \mathbf{D}(P_2)^\top\! P_1  \big[ C  \!-\!D_2 \big( R_2\!+\! D_2^{\top} P_2 D_2 \big)^{-1}\! \big( B_2^\top P_2\!+\! D_2^\top P_2 C \big)\big]   \big\}.\vspace{-1mm}
\end{align}
Combining this with the Riccati equation \eqref{P_2} for $P_2$ and the decoupled equation \eqref{P_1} for $P_1$, we obtain the ERE
\vspace{-1mm}
\begin{equation}\label{ERE}
\!\!\begin{cases}
\dot{P}_2=- ( A+B_1 \bar{\Theta}_1)^\top P_2 - P_2(A+B_1 \bar{\Theta}_1)- (C+D_1 \bar{\Theta}_1)^\top P_2 (C+D_1 \bar{\Theta}_1) - Q_2\\\ns\ds
\hspace{3.5em} +[P_2B_2+ (C+D_1 \bar{\Theta}_1)^\top P_2 D_2 ] [R_2+ D_2^\top P_2 D_2]^{-1} [B_2^\top P_2+ D_2^\top P_2 (C+D_1\bar{\Theta}_1) ], &\mbox{in }[0,T],\\\ns\ds
\dot{P}_1=-\big[ \mathbf{A}(\bar{\Theta}_1,P_2)^\top P_1+ P_1 \mathbf{A}(\bar{\Theta}_1,P_2) + \mathbf{C}(\bar{\Theta}_1,P_2)^\top P_1 \mathbf{C}(\bar{\Theta}_1,P_2) + Q_1+ \bar{\Theta}_1^\top R_1 \bar{\Theta}_1 \big], &\mbox{in }[0,T],\\
P_2(T)=G_2,\q P_1(T)=G_1.
\end{cases}\vspace{-1mm}
\end{equation}
The (linear) closed-loop equilibrium strategy is obtained by solving the ERE \eqref{bar-Theta1}--\eqref{ERE}.

\vspace{-2mm}

\subsection{Feedback Stackelberg equilibrium}\label{se3.2}

We now show that the closed-loop equilibrium strategy pair $(\bar{\Theta}_1, \Theta_2^*(\bar{\Theta}_1))$ constitutes a feedback Stackelberg equilibrium in the sense of \eqref{De-feed-Sta}. This establishes a connection between our time-consistent approach and the classical feedback Stackelberg concept.

First, we rewrite the ERE \eqref{ERE} in a more symmetric form. From \eqref{Th2*} and \eqref{nota}, we have  
\vspace{-1mm}
\begin{equation*}
\begin{aligned}
\mathbf{A}(\bar{\Theta}_1,P_2) = A + B_1 \bar{\Theta}_1 + B_2 \Theta_2^*(\bar{\Theta}_1), \qquad
\mathbf{C}(\bar{\Theta}_1,P_2) = C + D_1 \bar{\Theta}_1 + D_2 \Theta_2^*(\bar{\Theta}_1).
\end{aligned} \vspace{-1mm}
\end{equation*}
Substituting these into \eqref{ERE} yields the following compact form: 
\vspace{-1mm}
\begin{eqnarray}\label{ERE-2}
\!\!\begin{cases}\ns\ds
\dot{P}_2\!=\!-\big[ \mathbf{A}(\bar{\Theta}_1,P_2)^\top P_2\!+\! P_2 \mathbf{A}(\bar{\Theta}_1,P_2) \!+ \!\mathbf{C}(\bar{\Theta}_1,P_2)^\top P_2 \mathbf{C}(\bar{\Theta}_1,P_2) \!+\! Q_2 \!+\! \Theta_2^*(\bar{\Theta}_1)^\top R_2 \Theta_2^*(\bar{\Theta}_1) \big], &\mbox{in }[0,T],\\\ns\ds
\dot{P}_1\!=\!-\big[ \mathbf{A}(\bar{\Theta}_1,P_2)^\top P_1\!+\! P_1 \mathbf{A}(\bar{\Theta}_1,P_2) \!+\! \mathbf{C}(\bar{\Theta}_1,P_2)^\top P_1 \mathbf{C}(\bar{\Theta}_1,P_2) \!+\! Q_1\!+\! \bar{\Theta}_1^\top R_1 \bar{\Theta}_1 \big], &\mbox{in }[0,T],\\
P_2(T)=G_2,\q P_1(T)=G_1,
\end{cases} \vspace{-1mm}
\end{eqnarray}
with
\vspace{-1mm}
\begin{equation}\label{Th1-Th2}
\begin{cases}
\Theta_2^*(\bar{\Theta}_1) = -\big[ R_2 + D_2^{\top} P_2 D_2 \big]^{-1} \big[ B_2^\top P_2 + D_2^\top P_2 (C + D_1 \bar{\Theta}_1) \big],\\[6pt]
\bar{\Theta}_1 = -\big[ R_1 + \mathbf{D}(P_2)^\top P_1 \mathbf{D}(P_2) \big]^{-1} \Big\{ \mathbf{B}(P_2)^\top P_1 \\[4pt]
\qquad\q + \mathbf{D}(P_2)^\top P_1 \big[ C - D_2 (R_2 + D_2^{\top} P_2 D_2)^{-1} (B_2^\top P_2 + D_2^\top P_2 C) \big] \Big\}.
\end{cases}%\vspace{-1mm}
\end{equation}

Define the control strategies $(u^*, v^*)$ by
\begin{align}
u^*(t,x) &:= \bar{\Theta}_1(t) \, x, \label{u*}\\
v^*(t,x,u) &:= -\big[ (R_2 + D_2^{\top} P_2 D_2)^{-1} (B_2^\top P_2 + D_2^\top P_2 C) \big](t) \, x - \big[ (R_2 + D_2^{\top} P_2 D_2)^{-1} D_2^\top P_2 D_1 \big](t) \, u, \label{v*}
\end{align}
From these definitions, we immediately obtain the consistency relation
$v^*(t,x, u^*(t,x)) = \big[\Theta_2^*(\bar{\Theta}_1)\big](t) \, x$.

\begin{theorem}\label{th2}
Let Assumption \ref{H2} hold and assume that the ERE \eqref{ERE-2}--\eqref{Th1-Th2} admits a solution. Then the control pair $(u^*, v^*)$ defined by \eqref{u*}--\eqref{v*} is a feedback Stackelberg equilibrium as defined in \eqref{De-feed-Sta}.
\end{theorem}
\vspace{-4mm}
\begin{table}[htbp]
    \centering
    \renewcommand{\arraystretch}{1.5} % 撑开行距
    \caption{Comparison between feedback Stackelberg equilibrium and closed-loop equilibrium strategy}
    \label{Tab1}
    \renewcommand{\tabularxcolumn}[1]{m{#1}} % 保证所有列垂直居中
    
    % 注意看这里的定义，我加了首尾的竖线 | ，并且所有的横线都换成了 \hline
    \begin{tabularx}{\textwidth}{|>{\centering\arraybackslash}m{3cm} | >{\centering\arraybackslash}X |>{\centering\arraybackslash}X|}
        \hline
        \textbf{Feature} & \textbf{Feedback Stackelberg equilibrium (1.10)} & \textbf{Closed-loop equilibrium strategy (3.1)} \\
        \hline
        Follower's optimality & Locally optimal along the leader's specific equilibrium path & \textbf{Globally optimal} against \textbf{any} admissible leader control \\
        \hline
        Leader's optimality & Locally optimal given the follower's response & Time-consistent, no-regret equilibrium strategy \\
        \hline
        Time-consistency & By definition  & Proved via ERE and variational method \\
        \hline
        Robustness to leader's deviation & Low (follower cannot re-optimize) & High (follower always responds optimally) \\
        \hline
        Information requirement for follower & Needs to know $u^*$ (leader's specific strategy) & Needs only the feedback form of leader's control \\
        \hline
        Derivation method & Dynamic programming + static Stackelberg at Hamiltonian level & Variational method + ERE \\
        \hline
        Global well-posedness & Requires short horizon or control-independent diffusion \cite{Basar-Olsder-1999,Li-Yu-2018} & Established for arbitrary finite horizon under three scenarios (Theorem \ref{th4})\\
        \hline
    \end{tabularx}
\end{table}
% %
% \begin{table}[htbp]
% 	\centering
% 	\caption{Comparison between feedback Stackelberg equilibrium and closed-loop equilibrium strategy}
% 	\label{Tab1}
% 	\renewcommand{\tabularxcolumn}[1]{m{#1}}
% 	\begin{tabularx}{\textwidth}{>{\centering\arraybackslash}m{3cm} | >{\centering\arraybackslash}X |>{\centering\arraybackslash}X}
% 		\toprule
% 		\textbf{Feature} & \textbf{Feedback Stackelberg equilibrium (1.10)} & \textbf{Closed-loop equilibrium strategy (3.1)} \\
% 		\midrule
% 		Follower's optimality & Locally optimal along the leader's specific equilibrium path & \textbf{Globally optimal} against \textbf{any} admissible leader control \\
% 		\midrule
% 		Leader's optimality & Locally optimal given the follower's response & Time-consistent, no-regret equilibrium strategy \\
% 		\midrule
% 		Time-consistency & By definition  & Proved via ERE and variational method \\
% 		\midrule
% 		Robustness to leader's deviation & Low (follower cannot re-optimize) & High (follower always responds optimally) \\
% 		\midrule
% 		Information requirement for follower & Needs to know $u^*$ (leader's specific strategy) & Needs only the feedback form of leader's control \\
% 		\midrule
% 		Derivation method & Dynamic programming + static Stackelberg at Hamiltonian level & Variational method + ERE \\
% 		\midrule
% 		Global well-posedness & Requires short horizon or control-independent diffusion \cite{Basar-Olsder-1999,Li-Yu-2018} & Established for arbitrary finite horizon under three scenarios (Theorem \ref{th4})\\
% 		\bottomrule
% 	\end{tabularx}
% \end{table}

\begin{remark}\label{remark3.3}
In the feedback Stackelberg equilibrium \eqref{De-feed-Sta}, the follower optimizes only along the curve induced by the leader's specific strategy $u^*$, requiring advance knowledge of $u^*$--impractical if the follower does not know the leader's cost functional. In contrast, in our closed-loop equilibrium formulation (Definition \ref{De-closedloop}), the follower selects a globally optimal response to any admissible leader control, while the leader adopts a time-consistent, no-regret strategy. Thus, the follower's strategy is more robust and the leader retains flexibility. This conceptual distinction is summarized in Table \ref{Tab1}.
\end{remark}

\begin{remark}
The ERE \eqref{ERE-2}--\eqref{Th1-Th2} derived via our variational method coincides with the coupled Riccati equations from the HJB system of a static Stackelberg game at the Hamiltonian level \cite{Basar-Olsder-1999,Bensoussan-Chen-Sethi-2014,Huang-Shi-2024}. Thus, our approach offers an alternative derivation and, more importantly, establishes global well-posedness (Section \ref{se4})--a result previously unavailable except under restrictive short-time or control-independent diffusion assumptions.
\end{remark}

\vspace{-2mm}
\section{Global well-posedness of ERE \eqref{bar-Theta1}--\eqref{ERE}}\label{se4}

In this section, we present the global well-posedness of the ERE \eqref{bar-Theta1}--\eqref{ERE} for three distinct scenarios, including both one-dimensional and high-dimensional settings. The analysis relies on deriving suitable a priori estimates that prevent finite-time blow-up.  Detailed proofs are deferred to Subsections \ref{Appendix-pr-th3}--\ref{Appendix-pr-th4}; here we present the main statements and an intuitive discussion of the assumptions.

The following theorem shows that once a priori bounds are available, global existence and uniqueness follow via a standard continuation argument.

\begin{theorem}\label{th3}
Let Assumption \ref{H2} hold, and suppose that ERE \eqref{ERE-2}--\eqref{Th1-Th2} admits a priori estimates in the following sense: whenever \eqref{ERE-2}--\eqref{Th1-Th2} admits a solution $(P_1,P_2)\in C([0,T];\mathbb{S}_+^n\times\mathbb{S}_+^n)$, there exists a constant $\cC_{\mathrm{pri}}$, depending only on the system parameters, such that
\vspace{-1mm}
\begin{equation}\label{th3-eq1}
\sup_{s\in [0,T]}\left|\begin{pmatrix}
\operatorname{vec}(P_1(s))\\\operatorname{vec}(P_2(s))
\end{pmatrix} \right |\leq \mathcal{C}_{\mathrm{pri}}.\vspace{-1mm}
\end{equation}
Then ERE \eqref{ERE-2}--\eqref{Th1-Th2} admits a unique solution on the entire interval $[0,T]$.
\end{theorem}

We now present three scenarios under which the a priori estimate \eqref{th3-eq1} holds, each corresponding to a different structural condition on the follower's control authority over the drift and diffusion.

\begin{theorem}\label{th4}
Let Assumption \ref{H2} hold. Then ERE \eqref{ERE-2}--\eqref{Th1-Th2} admits a priori estimates under any of the following conditions:
%\vspace{-1mm}
\begin{itemize}
\item {\bf Case (i):} \text{One-dimensional case with $|D_2|>\delta>0$};

\item {\bf Case (ii):} \text{One-dimensional case with $D_2=0$};

\item {\bf Case (iii):} \text{High-dimensional case with $D_2=0$},
$
\mathrm{Rank}\big(B_2(t)\big)=n$,  $B_2(\cdot)$ and $R_2(\cdot)$ are continuously differentiable on $[0,T]$.
\end{itemize}
%\vspace{-1mm}
\end{theorem}
%
%\vspace{-2mm}
Cases (i) and (ii) are one-dimensional but allow  control-dependent diffusion; Case (iii) is high-dimensional but requires $D_2 = 0$ and a  full-rank condition on $B_2$.

\begin{remark}
The conditions in Theorem \ref{th4} are not mere technicalities. They ensure the leader can maintain a time-consistent, no-regret equilibrium by keeping the follower's influence on uncertainty predictable and controllable. We explain each case intuitively.

Case (i): One-dimensional, $|D_2| > \delta > 0$ (non‑degenerate diffusion effect of the follower). Here, the follower's control $v$ directly affects the diffusion via $D_2$, and this effect is uniformly non-degenerate. This allows the leader to invert the follower's influence on volatility and obtain a well-defined Riccati equation. Without non‑degeneracy, the follower could unpredictably amplify or suppress noise, preventing a stable, time-consistent strategy.

Case (ii): One-dimensional, $D_2 = 0$ (follower does not affect diffusion). The follower's control enters only the drift, eliminating any risk of additional volatility. The leader can then focus on steering the mean dynamics without worrying about the follower's effect on uncertainty. This is common in the feedback Stackelberg literature \cite{Basar-Olsder-1999,Bensoussan-Chen-Sethi-2014,Huang-Shi-2024}, but we remove the short-time restriction and prove global well‑posedness for any finite horizon in the one‑dimensional setting.

Case (iii): High-dimensional, $D_2 = 0$, $\operatorname{Rank}(B_2(t)) = n$ (full drift controllability). In higher dimensions, eliminating the follower's diffusion influence is insufficient; the follower must also be able to influence every direction of the state space through the drift. The rank condition ensures full controllability, allowing the leader to delegate drift stabilization to the follower while focusing on her own objective. Without it, some state components would be uncontrollable, exposing the leader to unpredictable dynamics and breaking time-consistency.

From a management  or economics perspective, these conditions reflect a hierarchy where the leader anticipates the follower's globally optimal response, and the follower either avoids introducing extra risk ($D_2 = 0$) or acts predictably even when influencing risk ($|D_2| > \delta$). In high-dimensional problems, the follower must influence every aspect of the system (full rank) for the leader's strategy to remain time-consistent regardless of the initial state. These are natural requirements for a well-functioning hierarchy in a stochastic environment.
\end{remark}

\begin{remark}
Combining Theorems \ref{th3} and \ref{th4}, the ERE \eqref{ERE-2}--\eqref{Th1-Th2} admits a unique global solution on any finite horizon $T>0$ in Cases (i)--(iii). Consequently, the leader's closed-loop equilibrium strategy exists (Theorem \ref{th1}), or equivalently, a feedback Stackelberg equilibrium exists (Theorem \ref{th2}). We highlight two points:

1. {\bf Regardless of time-inconsistency}.
To our knowledge, general conditions for pre-committed closed-loop Stackelberg solutions are unavailable except for special examples \cite{Bensoussan-Chen-Sethi-2015,Hernandez-etal-2024}.

2. {\bf Taking time-inconsistency into account}.
In contrast, existing results on feedback Stackelberg equilibria are largely restricted to control-independent diffusion and require $T$ sufficiently small \cite{Basar-Olsder-1999,Huang-Shi-2024}.
\end{remark}

\vspace{-1mm}
\section{Application}\label{sec5}

In this section, we apply Theorem \ref{th4} to an asset management problem with two agents operating under asymmetric information. Numerical simulations validate the theoretical predictions. The model is particularly well-suited for Stackelberg LQ problems where control enters the diffusion term.

\vspace{-1mm}

\subsection{Model and solutions}\label{sec5.1}

Consider an asset $X$ managed by two agents, both of whom derive benefit from its terminal value $X(T)$. The financial market consists of a risk-free bond $S^0$ and two risky stocks $S^1, S^2$ governed by: 
\vspace{-1mm}
\begin{equation*}
d S^0 = r S^0 ds, \quad d S^i = \mu_i S^i ds + \sigma_i S^i dW(s), \quad i=1,2,\vspace{-1mm}
\end{equation*}
where $r > 0$ is the risk-free rate, and $\mu_i, \sigma_i > 0$ are the appreciation and volatility rates, respectively. Let $u_i$ be the amount Agent $i$ invests in stock $i$. The dynamics of the property $X$ follows:
\vspace{-2mm}
\begin{equation}\label{model}
\begin{cases}\ds
dX = \Big[ r X + \sum_{i=1}^2 (\mu_i - r)u_i \Big] ds + \sum_{i=1}^2 \sigma_i u_i dW(s), \q \mbox{in } [0,T], \\
\ns\ds X(0) = x_0.
\end{cases}\vspace{-1mm}
\end{equation}
Motivated  by G\^{a}rleanu and Pedersen \cite{Garleanu-Pedersen-2016},  we adopt the following conventions:
\begin{enumerate}
\item[(i)] \textbf{Terminal Goal:} Both agents aim to minimize the variance from a target $z\in \dbR$, where $z \ge x_0 e^{rT}$. The objective is to minimize $(X(T) - z)^2$.
\item[(ii)] \textbf{Transaction Costs:} To reflect risk-based compensation, Agent $i$ incurs a cost measured by $\sigma_i u_i^2(s)$.
\item[(iii)] \textbf{Information Asymmetry:} Agent 1 (Leader) possesses full information $(\mu_1, \mu_2, \sigma_1, \sigma_2)$, whereas Agent 2 (Follower) only observes her own market parameters $(\mu_2, \sigma_2)$.
\end{enumerate}
Under these conventions, the cost functional for Agent $i$ at each initial pair $(t,\xi)$ is
\vspace{-1mm}
\begin{equation*}
\mathcal{J}_i(t, \xi; u_1, u_2) = \mathbb{E}_t \[ \int_{t}^{T} \sigma_i u_i^2(s) ds + (X(T) - z)^2 \].\vspace{-1mm}
\end{equation*}

To deal with the above model, we take the following steps:

\ss

{\bf Step 1.} In this step, we introduce an equivalent auxiliary control problem, and reformulate the leader's control problem accordingly.

By defining the shifted state $\wt X(s) := X(s) - z e^{-r(T-s)}$, we introduce the following equivalent auxiliary control problem:
\vspace{-1mm}
\begin{equation*}
\begin{cases}
\begin{cases}\ds
d\wt X=\Big[ r \wt X+\sum_{i=1}^{2}(\mu_i-r)u_i \Big]ds+   \sum_{i=1}^{2}\sigma_i u_i dW(s), &\q \mbox{in } [t,T],\\\ns\ds
\wt X(t)=\xi-z e^{-r(T-t)},
\end{cases}\\\ns\ds
\wt \cJ_i\Big(t,\xi-z e^{-r(T-t)};u_1,u_2\Big):=\mE_t \[\int_{t}^{T}  \langle \sigma_i u_i,u_i \rangle  ds + \langle \wt X(T),\wt X(T) \rangle\],\q (i=1,2).
\end{cases}
\end{equation*}
It is clear that
\begin{equation}\label{ex-eq1}
\cJ_i(t,\xi;u_1,u_2)=\wt \cJ_i\Big(t,\xi-z e^{-r(T-t)};u_1,u_2\Big),\q i=1,2.
\end{equation}

Now we focus on the auxiliary control problem. Following the procedure in Subsection \ref{se-2.1}, for Agent 1 choosing $u_1=\Theta_1 \wt X$, we solve the optimal control problem for Agent 2, that is
\vspace{-1mm}
$$
\min_{u_2}\wt \cJ_2\Big(t,\xi-z e^{-r(T-t)};\Theta_1\wt X,u_2\Big).\vspace{-1mm}
$$
From Proposition \ref{pro2.1}, we have $u_2^*=\Theta_2^*(\Theta_1)\wt X$,
where
\vspace{-1mm}
$$
\Theta_2^*(\Theta_1)=-\big[\sigma_2+\sigma_2^2 P_2\big]^{-1}\big[(\mu_2-r)P_2+ \sigma_2 P_2 \sigma_1 \Theta_1\big],\vspace{-1mm}
$$
$P_2$ solves the following equation
\vspace{-1mm}
\begin{equation*}
\begin{cases}
\dot{P_2}+2[r+(\mu_1-r)\Theta_1]P_2+ \sigma_1^2\Theta_1^2P_2-[P_2(\mu_2-r)+\sigma_1 \Theta_1 P_2 \sigma_2]^2 [\sigma_2+\sigma_2^2P_2]^{-1}=0,\q \mbox{in } [t,T],\\\ns\ds
P_2(T)=1,
\end{cases}\vspace{-1mm}
\end{equation*}
and $\widetilde{X}$ is the unique solution of the closed-loop system under $u_2^* = \Theta_2^*(\Theta_1)\widetilde{X}$. Consequently, Agent 1's control problem becomes 
\vspace{-1mm}
\begin{eqnarray}\label{auxi-agent1}
\!\!\!\begin{cases}
\begin{cases}
d\wt X  = \big\{ \big[r+(\mu_1 - r)\Theta_1\big] \wt X  +  (\mu_2-r)\Theta_2^*(\Theta_1)\wt X \big\} ds + \big\{  \sigma_1 \Theta_1 \wt X + \sigma_1 \Theta_2^*(\Theta_1)\wt X  \big\} dW(s),  &  \mbox{in }[t,T],\\\ns\ds
\dot{P_2}=-2\big[r + (\mu_1 - r)\Theta_1\big]P_2 -  \sigma_1^2\Theta_1^2P_2 + \big[P_2(\mu_2 - r) + \sigma_1 \Theta_1 P_2 \sigma_2\big]^2 \big[\sigma_2 + \sigma_2^2P_2\big]^{-1},  &  \mbox{in }[t,T], \\\ns\ds
\wt X(t)=\xi-z e^{-r(T-t)},\q P_2(T)=1,
\end{cases}\\\ns\ds
\wt \cJ_1\Big(t, \xi-z e^{-r(T-t)}; \Theta_1 \wt X, \Theta_2^*(\Theta_1)\wt X\Big) = \mathbb{E}_t \[ \int_{t}^{T} \sigma_1 \Theta_1^2 \wt X^2 ds + \wt X(T)^2 \].
\end{cases}\vspace{-1mm}
\end{eqnarray}

Motivated from Definition \ref{De-closedloop}, we introduce the following definition for \eqref{auxi-agent1}:%\vspace{-1mm}

\begin{definition}\label{de-5.1}
A function $\bar{\Theta}_1 \in L^\infty(0,T;\dbR)$ is called a closed-loop equilibrium strategy for Problem \eqref{auxi-agent1} if for any $ (t,\xi,\dbV_1)  \in [0,T)\times L_{\cF_t}^2(\Omega;\dbR)\times L^\infty(0,T;\dbR),$
the following variational inequality holds:
\vspace{-1mm}
\begin{equation}\label{ex-eq2}
\liminf_{\e\to 0+} \frac{\wt \cJ_1\Big(t,\xi-z e^{-r(T-t)};\bar{\Theta}^\e_1 \wt X^\e,\Theta_2^*(\bar{\Theta}^\e_1)\wt X^\e\Big)-
\wt \cJ_1\Big(t,\xi-z e^{-r(T-t)};\bar{\Theta}_1 \wt X,\Theta_2^*(\bar{\Theta}_1)\wt X\Big)}{\e}\geq 0,  \q \as
\end{equation}
where
$\bar{\Theta}_1^\e(s)=\bar{\Theta}_1(s)+1_{[t,t+\e]}\dbV(s)$,
and $\wt X^\e(\cd)$ denotes the solution of the state equation in \eqref{auxi-agent1} corresponding to the perturbed control pair $\big( \bar{\Theta}^\e_1 \wt X^\e,\Theta_2^*(\bar{\Theta}^\e_1) \wt X^\e\big)$.
\end{definition}

{\bf Step 2.} In this step, we derive $\bar{\Theta}_1$ from Definition \ref{de-5.1} and obtain the time-consistent equilibrium controls for both agents.

Since $\xi\in L_{\cF_t}^2(\Omega;\dbR)$ is arbitrary in Definition \ref{de-5.1}, $\xi-z e^{-r(T-t)}$ is also arbitrary,
and thus we can replace $\xi-z e^{-r(T-t)}$ in \eqref{ex-eq2} by any  $\eta \in L_{\cF_t}^2(\Omega;\dbR)$. Then Problem \eqref{auxi-agent1} and Definition \ref{de-5.1} coincide with Problem \eqref{ConPro-leader} and Definition \ref{De-closedloop}, respectively. Setting $A = r$, $C = 0$, $B_i = \mu_i - r$, $D_i = R_i = \sigma_i$, $G_i = 1$ ($i = 1,2$) in \eqref{ERE} yields the following ERE:
\vspace{-1mm}
\begin{align}\label{ex-ERE}
\begin{cases}
\dot{P_2}=-2\big[r+(\mu_1-r)\bar{\Theta}_1\big]P_2- \sigma_1^2\bar{\Theta}_1^2P_2+\big[P_2(\mu_2-r)+\sigma_1 \bar{\Theta}_1 P_2 \sigma_2\big]^2\big[\sigma_2+\sigma_2^2 P_2\big]^{-1}, & \mbox{in } [0,T],\\\ns\ds
\dot{P_1}=- 2P_1\big\{r+(\mu_1-r)\bar{\Theta}_1-(\mu_2-r) (\sigma_2+\sigma_2^2P_2)^{-1} \big[(\mu_2-r)P_2+\sigma_2P_2\sigma_1\bar{\Theta}_1\big] \big\}\\\ns\ds
\qq\q - P_1\big\{\sigma_1 \bar{\Theta}_1 -\sigma_2 \big(\sigma_2+\sigma_2^2 P_2\big)^{-1}\big[(\mu_2-r)P_2+\sigma_2 P_2\sigma_1\bar{\Theta}_1\big] \big\}^2 - \sigma_1 \bar{\Theta}_1^2 ,& \mbox{in } [0,T],\\\ns\ds
P_2(T)=1,\q P_1(T)=1,
\end{cases}\vspace{-1mm}
\end{align}
and
\vspace{-2mm}
\begin{equation}\label{ex-Th1}
\begin{aligned}
\bar{\Theta}_1&=-\Big\{\sigma_1+P_1 \big[\sigma_1-\sigma_2(\sigma_2+\sigma_2^2P_2)^{-1}\sigma_2P_2\sigma_1 \big]^2\Big\}^{-1} \Big\{P_1 \big[(\mu_1-r)-(\mu_2-r)(\sigma_2+\sigma_2^2 P_2)^{-1} \sigma_2 P_2 \sigma_1\big]\\&
\qq-P_1\big[\sigma_1-\sigma_2(\sigma_2+\sigma_2^2P_2)^{-1}\sigma_2P_2\sigma_1 \big]\big[\sigma_2 (\sigma_2+\sigma_2^2P_2)^{-1}(\mu_2-r)P_2 \big] \Big\}.%\vspace{-2mm}
\end{aligned}	\vspace{-1mm}
\end{equation}

From Cases (i) and (ii) of Theorem \ref{th4}, in the one-dimensional case with constant system parameters, the ERE \eqref{ERE} always admits a unique solution under positivity assumptions. Hence, \eqref{ex-ERE} admits a unique solution $(P_1, P_2)$, and $\bar{\Theta}_1$ is given by \eqref{ex-Th1}. Finally, letting both agents choose
\vspace{-1mm}
\begin{equation}\label{u_1u_2}
u_1^*(s,X(s))=\bar{\Theta}_1(s) \big[X(s)-ze^{-r(T-s)}\big],\qq
u_2^*(s,X(s))= \Theta_2^*(\bar{\Theta}_1)(s)\big[ X(s)-ze^{-r(T-s)}\big],\vspace{-1mm}
\end{equation}
they obtain the time-consistent equilibrium control for the original control problem.

\begin{remark}
Although our analysis is set within a Stackelberg LQ framework, it offers useful insights for more general nonlinear problems. In many applications where the state dynamics admit a local linear approximation and the cost functional is a quadratic one, the resulting nonlinear problem can be treated--at least approximately--within an LQ setting. In such cases, our results remain applicable as useful approximations.
\end{remark}

\vspace{-3mm}

\subsection{Numerical simulations}

We implement a  backward Euler-type scheme for the ERE \eqref{ex-ERE}. The time horizon $[0, T]$ is uniformly discretized into $N$ subintervals with grid points $t_i = i\Delta t$, where $\Delta t = T/N$. The recursion is initialized with terminal conditions $P_{1,N}=1$, $P_{2,N}=1$.

For each backward step $i = N, N-1, \ldots, 1$:

1.  Compute intermediate quantities:
\vspace{-1mm}
$$
I_{2,i} = (\sigma_2 + \sigma_2^2 P_{2,i})^{-1}, \qquad
A_{1,i} = \sigma_1 - \sigma_2 I_{2,i} \sigma_2 P_{2,i} \sigma_1.\vspace{-1mm}
$$

2.  Update feedback strategies:
\vspace{-1mm}
$$
\left\{
\begin{aligned}
\ds	\bar{\Theta}_{1,i-1}&  = -\big(\sigma_1 + P_{1,i} A_{1,i}^2\big)^{-1}\big\{ P_{1,i}\big[(\mu_1 - r) - (\mu_2 - r) I_{2,i} \sigma_2 P_{2,i} \sigma_1\big] - P_{1,i} A_{1,i} \big[\sigma_2 I_{2,i} (\mu_2 - r) P_{2,i}\big]\big\}, \\
\ns\ds	\Theta_{2,i-1}^* & = -I_{2,i}\big[(\mu_2 - r) P_{2,i} + \sigma_2 P_{2,i} \sigma_1 \bar{\Theta}_{1,i-1}\big].
\end{aligned}
\right.\vspace{-1mm}
$$

3.  Compute time derivatives:
\vspace{-1mm}
\begin{equation*}
	\begin{cases}
		\begin{aligned}
			\dot{P}_{1,i} &= -\big\{ 2P_{1,i}\big[r+(\mu_1-r)\bar{\Theta}_{1,i-1} + (\mu_2-r)\Theta_{2, i-1}^* \big]  + P_{1,i} \big[\sigma_1 \bar{\Theta}_{1,i-1} + \sigma_2 \Theta_{2, i-1}^* \big]^2 + \sigma_1 \bar{\Theta}_{1,i-1}^2 \big\}, \\\ns\ds
			\dot{P}_{2,i} &= -2\big[r+(\mu_1-r)\bar{\Theta}_{1,i-1}\big]P_{2,i}- \sigma_1^2\bar{\Theta}_{1,i-1}^2 P_{2,i} + \big[P_{2,i}(\mu_2-r)+\sigma_1 \bar{\Theta}_{1,i-1} P_{2,i} \sigma_2\big]^2 I_{2,i}.
		\end{aligned}
	\end{cases}\vspace{-1mm}
\end{equation*}

4.  Backward Euler update:
$$
P_{j,i-1} = P_{j,i} - \dot{P}_{j,i} \Delta t, \quad j = 1,2.\vspace{-1mm}
$$

This recursive procedure yields the discrete trajectories of the ERE \eqref{ex-ERE}, which are then used in the forward simulation of the wealth process.

We simulate the Stackelberg LQ model from Subsection \ref{sec5.1} with parameters given in Table \ref{tab:params}. To comprehensively illustrate the dynamics, we present results for three distinct sample paths under the feedback Stackelberg equilibrium \eqref{u_1u_2}. Figure \ref{Fig1_P} shows the deterministic solutions $P_1(s)$ and $P_2(s)$ of \eqref{ex-ERE}, which govern the feedback strategies of the leader and follower, respectively. Figure \ref{Fig2_X} displays the stochastic realizations of the wealth process $X(s)$ for the three paths; as expected, all trajectories converge to the terminal wealth target $z$ (black dotted line). Finally, Figure \ref{Fig3_U} compares the dynamic asset allocations $u_1$ (leader, solid blue) and $u_2$ (follower, dashed red) for each of the three paths.

\begin{table}[htbp]
\centering
\caption{Simulation Parameters and Economic Meaning}
\label{tab:params}
\begin{tabular}{lcl}
\hline
\textbf{Parameter} & \textbf{Value} & \textbf{Economic Rationale} \\ \hline
Initial Wealth $x_0$ & 100.0 & Initial principal at $t=0$ \\
Wealth Target $z$ & 200.0 & Desired terminal payoff at $T=10$ \\
Risk-free Rate $r$ & 0.03 & Baseline yield of riskless assets \\
Drift $\mu_1, \mu_2$ & 0.08, 0.10 & Anticipated returns of risky assets \\
Volatility $\sigma_1, \sigma_2$ & 0.15, 0.19 & Standard deviation of market fluctuations \\
Time Horizon $T$ & 10.0 & Strategic planning interval  \\ \hline
\end{tabular}
\end{table}
%\vspace{-1mm}
\begin{figure}[htbp]
\centering
\includegraphics[width=0.5\textwidth]{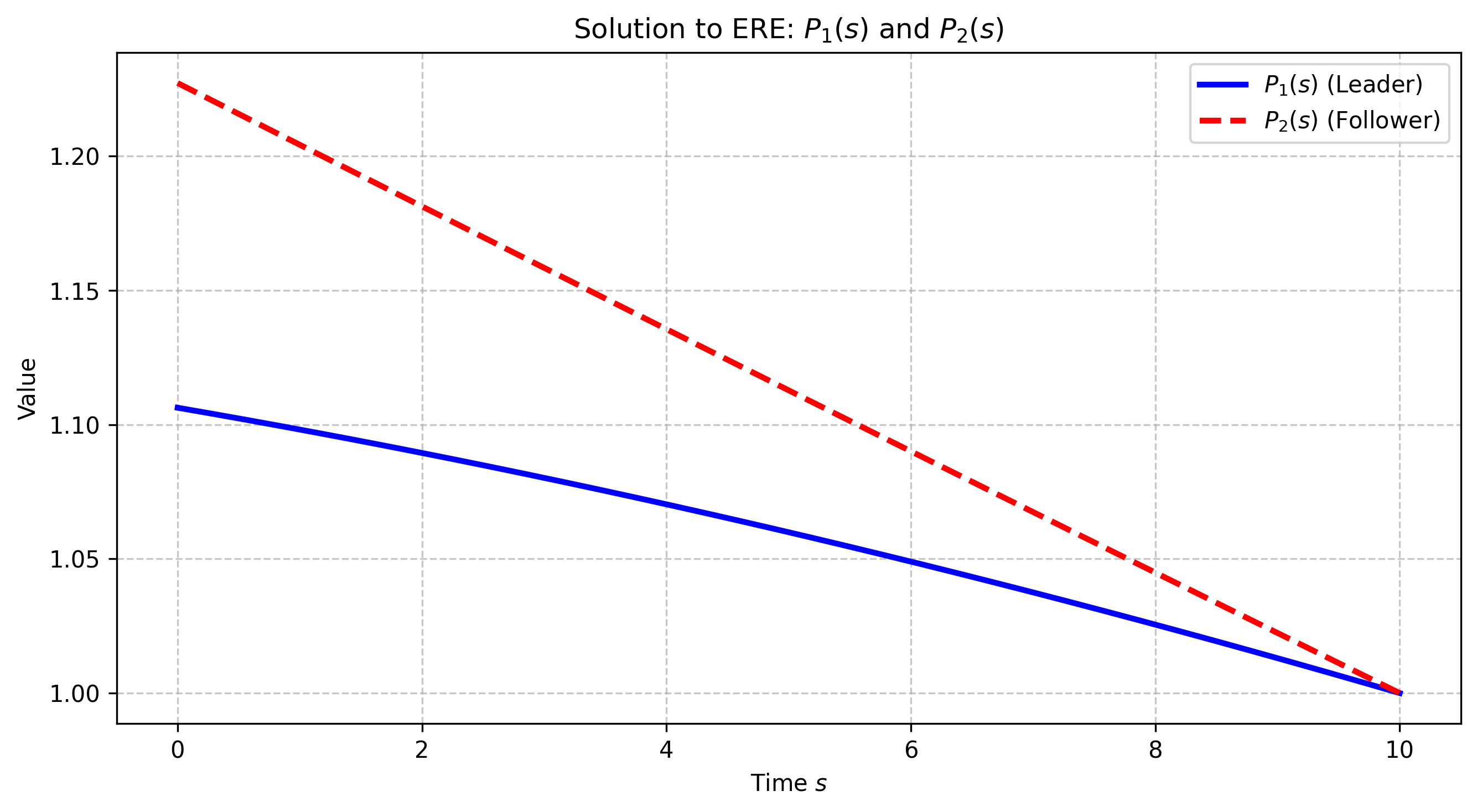} % Ensure filename matches Python output
\caption{\small Deterministic solutions $P_1(s)$ and $P_2(s)$ governing the feedback Stackelberg equilibrium.}
\label{Fig1_P}
\end{figure}

Figure 2 displays three sample paths of the wealth process $X(s)$ under the closed-loop equilibrium strategies. Several observations are in order:

1. {\bf Convergence to the target}. All three paths start from the same initial condition, increase over time, and approach the terminal target $z$ as $s \to T$. This shows that the equilibrium strategies successfully steer wealth toward the goal despite stochastic volatility.

2. {\bf Robustness to noise}.  Although the paths temporarily diverge due to different realizations of the Wiener process, they all converge to $z$ within a narrow band at $s = T$. This indicates that the closed-loop equilibrium strategy is robust to market uncertainty.

3. {\bf No overshooting}.  Unlike some pre-committed strategies that may overshoot and then correct, the equilibrium paths increase monotonically (on average) without significant oscillations. This reflects the time-consistent nature of the strategy: the leader never deviates from a plan that would later require regret-inducing corrections.

%\vspace{-2mm}
\begin{figure}[H]
\centering
\includegraphics[width=0.5\textwidth]{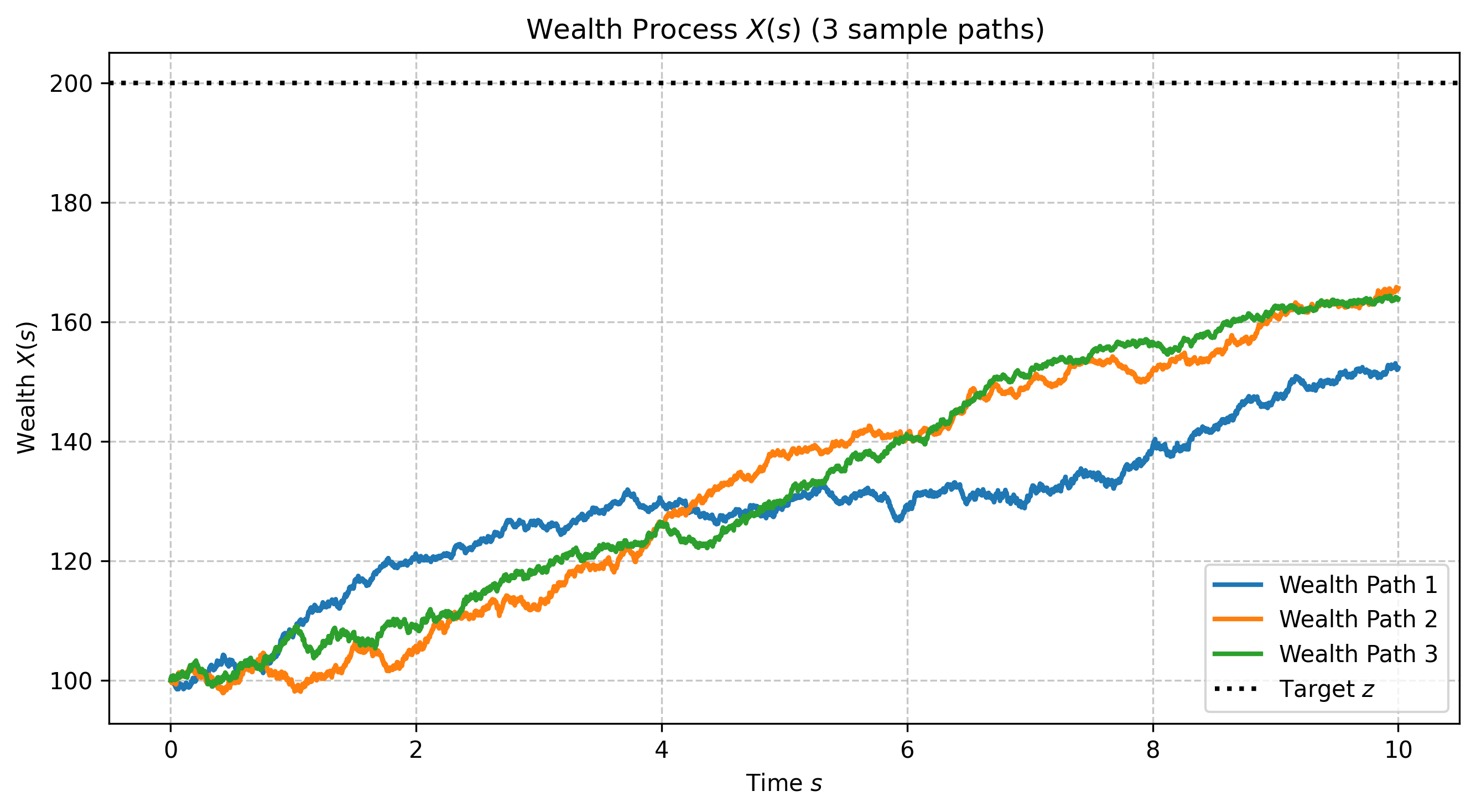} % Ensure filename matches Python output
\caption{\small Three sample paths of the wealth process $X(s)$ converging towards the terminal target $z$.}
\label{Fig2_X}
\end{figure}

Figure 3 compares the control trajectories $u_1(s)$ (leader) and $u_2(s)$ (follower) across  three sample paths.
Note that in the simulation, we have ${\mu_1}/{\sigma_1} > {\mu_2}/{\sigma_2}$. One might therefore expect the leader to invest more aggressively than the follower, i.e., $u_1(s) > u_2(s)$. However, the simulation results reveal the opposite pattern: $u_1(s) < u_2(s)$ for all $s$. This seemingly counterintuitive outcome can be attributed to the strategic interaction between the leader and the follower in a Stackelberg framework:

1. {\bf Leader's strategic leverage}. The magnitude of $u_1(s)$ is consistently smaller than that of $u_2(s)$ throughout the horizon. This reflects the leader's advantageous position: because she anticipates the follower's rational response and internalizes the follower's actions, she can strategically rely on the follower's effort. Consequently, the leader takes more conservative investment positions, achieving her goal with less control exertion.

2. {\bf Follower's reactive burden}. The follower's control $u_2(s)$ is clearly larger in magnitude. This is consistent with the follower's role: given the leader's strategy, the follower must take more aggressive investment positions to optimize her own cost. Lacking the hierarchical power to shape the system, she is forced to bear the primary control burden.

3. {\bf Path dependence}. The three sample paths of each control differ slightly, reflecting dependence on the realized noise. However, the feedback form $u_i(s) = \Theta_i(s)(X(s) - z e^{-r(T-s)})$ ensures that the control at each instant is adapted to the current state, so the variations across paths are natural and do not indicate instability.

\vspace{-1mm}
 \begin{figure}[H]
\centering
\includegraphics[width=0.5\textwidth]{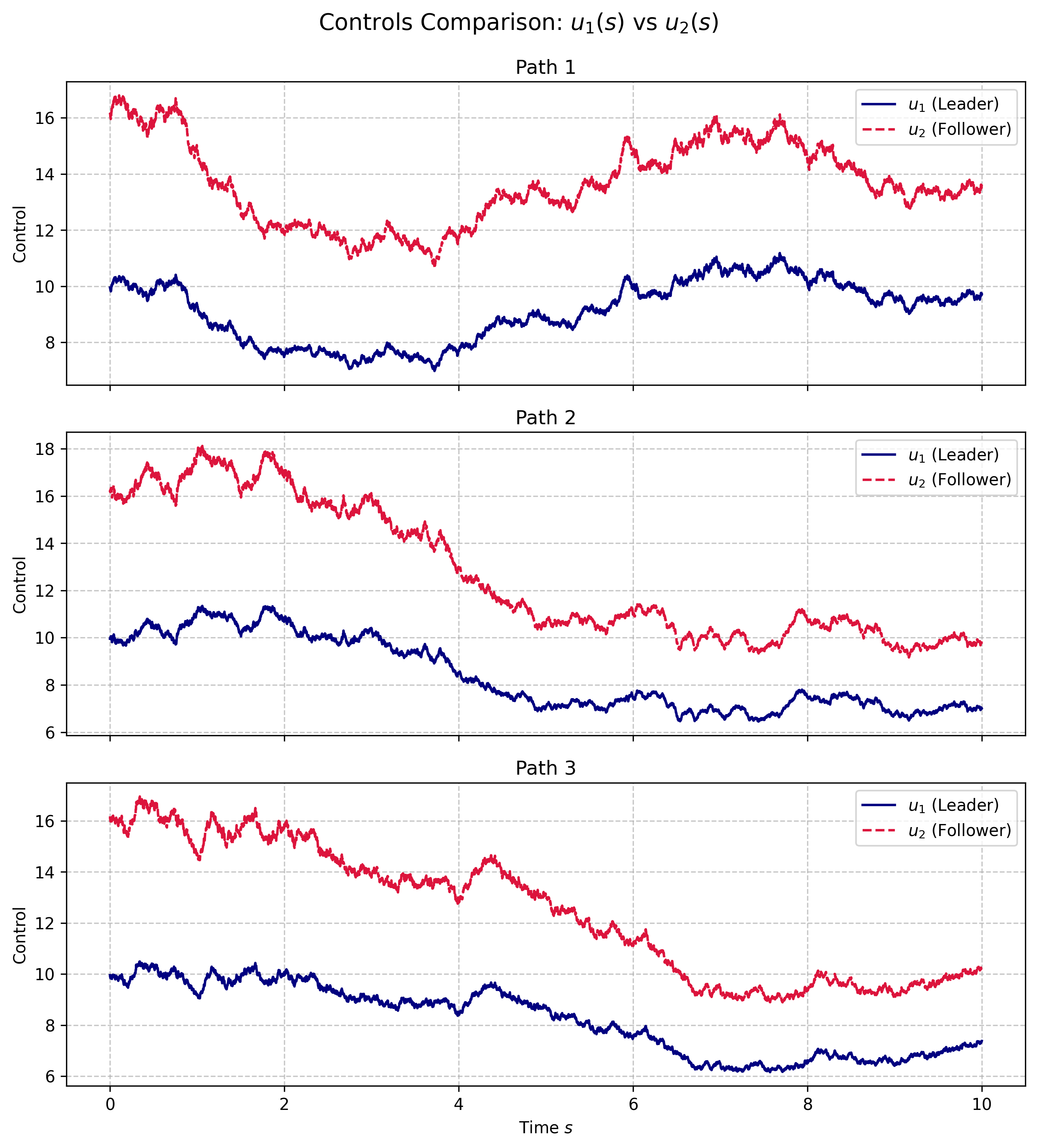} % Ensure filename matches Python output
\caption{\small Comparison of control trajectories $u_1(s)$ (leader) and $u_2(s)$ (follower) across three sample paths.}
\label{Fig3_U}
\end{figure}

\begin{remark}
The numerical simulation presented here is merely an approximation, and the observed phenomena are illustrative rather than mathematically rigorous. A full justification--including a stability analysis of the ERE and a precise interpretation of the approximate equilibrium strategy--is required but falls outside the scope of this paper, and is left for future research.
\end{remark}

\vspace{-2mm}

\section{Proof of the main results}\label{Appendix}

%\vspace{-1mm}

This section collects the proofs of Theorems \ref{th1}, \ref{th2}, \ref{th3}, and \ref{th4}.

%\vspace{-2mm}

\subsection{Proof of Theorem \ref{th1}}\label{Appendix-pr-th1}
\begin{proof}
We divide the proof into two steps:

{\bf Step 1.} In this step, we derive essential estimates for the original system and its perturbed counterpart.

For any initial pair $(t,\xi)\in [0,T) \times L_{\cF_t}^2(\Omega;\dbR^n) $, consider a control strategy $\Theta_1\in L^\infty(0,T;\dbR^{m_1 \times n})$ and its perturbation
\begin{equation}\label{pr-th1-eq1}
\Theta_1^\e := \Theta_1+ 1_{[t,t+\e]} \dbV \text{ with } \dbV(\cd)\in L^\infty(0,T;\dbR^{m_1 \times n}).
\end{equation}
Substituting these into the state equation \eqref{ConPro-leader} yields $P_2$, $P_2^\varepsilon$, and
\begin{equation}
\begin{cases}
dX = \mathbf{A} X ds+ \mathbf{C} X dW(s), & \mbox{in } [t,T],\\
X(t)=\xi,
\end{cases}\q
\begin{cases}
dX^\e= \mathbf{A}^\e X^\e ds+ \mathbf{C}^\e X^\e dW(s), & \mbox{in } [t,T],\\
X^\e(t)=\xi,
\end{cases}
\end{equation}
where
\begin{equation*}
\begin{aligned}
&\mathbf{A}:= \mathbf{A}(\Theta_1,P_2),\q \mathbf{C}:= \mathbf{C}(\Theta_1,P_2),\\\ns\ds
&\mathbf{A}^\e:=\mathbf{A}(\Theta^\e_1,P^\e_2)= A+ B_1 \Theta^\e_1 - B_2  \big[ R_2+ D_2^{\top} P^\e_2 D_2 \big]^{-1} \big[ B_2^\top P^\e_2+ D_2^\top P^\e_2 (C+D_1\Theta^\e_1) \big],\\ \ns\ds
& \mathbf{C}^\e:= \mathbf{C}( \Theta_1^\e,P_2^\e)= C+D_1 \Theta^\e_1  -D_2 \big[ R_2+ D_2^{\top} P^\e_2 D_2 \big]^{-1} \big[ B_2^\top P^\e_2+ D_2^\top P^\e_2 (C+D_1\Theta^\e_1) \big].
\end{aligned}
\end{equation*}
Define $X_0^\varepsilon = X^\varepsilon - X$, $\mathbf{A}_0^\varepsilon = \mathbf{A}^\varepsilon - \mathbf{A}$, $\mathbf{C}_0^\varepsilon = \mathbf{C}^\varepsilon - \mathbf{C}$. Then 
\begin{equation}
\begin{cases}
d X^\e_0 =\big[ \mathbf{A} X^\e_0+  ( \mathbf{A}^\e -\mathbf{A} ) X^\e  \big]ds + \big[  \mathbf{C} X^\e_0+ (\mathbf{C}^\e -\mathbf{C})X^\e \big]dW(s),\q \mbox{in } [t,T],\\
X^\e_0(t)=0.
\end{cases}
\end{equation}
From \eqref{pr-th1-eq1}, $\Theta_1^\varepsilon(s)=\Theta_1(s)$ for $s\in(t+\varepsilon,T]$. Since $P_2$ is uniquely determined by $\Theta_1$ via \eqref{P_2}, we obtain 
\begin{equation}\label{pr-th1-est6}
P^\e_2(s)=P_2(s),\q
\mathbf{A}^\e(s)=\mathbf{A}(s),\q
\mathbf{C}^\e(s)=\mathbf{C}(s),\q \forall s\in (t+\e,T],
\end{equation}
which readily leads to the following estimates for $P_2$ and $P_2^\varepsilon$:
\begin{equation*}\label{pr-th1-est1}
\begin{aligned}
|P_2(\cd)|_{C([t,T];\dbR^{n\times n})}\leq \cC (\Theta_1),\q |P_2^\e(\cd)|_{C([t,T];\dbR^{n\times n})}\leq \cC (\Theta_1,\dbV),\q |P_2^\e(\cd)-P_2(\cd)|_{C([t,T];\dbR^{n\times n})}\leq \cC (\Theta_1,\dbV) \e.
\end{aligned}
\end{equation*}
The boundedness of $P_2$ and $P_2^\varepsilon$, together with standard $L^2$-estimates for linear SDEs, implies
\begin{equation}\label{pr-th1-est2}
\mE_t \Big[ \sup_{s\in [t,T]} |X(s)|^2  \Big]\leq \cC (\Theta_1) |\xi|^2, \quad
\mE_t \Big[ \sup_{s\in [t,T]} |X^\e(s)|^2  \Big]\leq \cC (\Theta_1,\dbV) |\xi|^2.%\vspace{-2mm}
\end{equation}
To obtain estimates for $X^\varepsilon - X$, we first estimate $\mathbf{A}^\varepsilon - \mathbf{A}$ and $\mathbf{C}^\varepsilon - \mathbf{C}$. We decompose the difference:
\begin{equation*}
\mathbf{A}^\e - \mathbf{A} = \big[ \mathbf{A}(\Theta^\e_1, P^\e_2) - \mathbf{A}(\Theta^\e_1, P_2) \big] + \big[ \mathbf{A}(\Theta^\e_1, P_2) - \mathbf{A}(\Theta_1, P_2) \big].
\end{equation*}
The first term is bounded by $\cC(\Theta_1,\dbV)|P^\e_2 - P_2| \leq \cC(\Theta_1,\dbV)\e$.
Since the second term  is linear in $\Theta_1$ (see \eqref{nota}),  we have $\mathbf{A}(\Theta^\e_1, P_2) - \mathbf{A}(\Theta_1, P_2) = \mathbf{B}(P_2)\dbV 1_{[t,t+\e]}$.
Thus, the following estimate holds:
\begin{equation}\label{pr-th1-est3}
\big|\mathbf{A}^\e-\mathbf{A} - \mathbf{B}(P_2)\dbV 1_{[t,t+\e]} \big| \leq \cC(\Theta_1,\dbV) \e.%\vspace{-1mm}
\end{equation}
Similarly, we obtain
\begin{equation}\label{pr-th1-est4}
\big|\mathbf{C}^\e-\mathbf{C} - \mathbf{D}(P_2)\dbV 1_{[t,t+\e]} \big| \leq \cC(\Theta_1,\dbV) \e.
\end{equation}
Using standard $L^2$-estimates for linear SDEs together with \eqref{pr-th1-est2}--\eqref{pr-th1-est4}, we have 
\vspace{-2mm}
\begin{equation}\label{pr-th1-est5}
	\begin{aligned}
\mE_t \Big[ \sup_{s\in [t,T]} |X^\e_0(s)|^2 \Big] &\leq \cC (\Theta_1) \mE_t \int_t^T \Big( |(\mathbf{A}^\e - \mathbf{A})X^\e|^2 + |(\mathbf{C}^\e - \mathbf{C})X^\e|^2 \Big) ds \\
&\leq \cC (\Theta_1) \mE_t \Big[\sup_{s\in [t,T]} |X^\e(s)|^2 \int_t^{t+\e} \cC(\Theta_1,\dbV) \big(1_{[t,t+\e]}+1_{[t,t+\e]}\e\big)  ds \\ &\q  + \sup_{s\in [t,T]} |X^\e(s)|^2 \int_t^{t+\e} \cC(\Theta_1,\dbV) \big(1_{[t,t+\e]}+1_{[t,t+\e]}\e\big)  ds \Big] \\
&\leq \cC (\Theta_1,\dbV) |\xi|^2 \e.
\end{aligned}
\end{equation}

{\bf Step 2.} In this step, we use decoupling techniques to precisely represent the variation of the cost functional and derive the equivalent characterization of the closed-loop equilibrium strategy from it.

Direct computation yields
\vspace{-1mm}
\begin{align*}
&\cJ_1(t,\xi; \Theta_1^\e X^\e,\Theta_2^*({\Theta}^\e_1) X^\e)- \cJ_1(t,\xi; \Theta_1 X,\Theta_2^*({\Theta}_1) X )
\\&=\frac{1}{2}\mE_t \[ \int_{t}^{T} \big[  \langle Q_1 X^\e_0,X^\e_0  \rangle + \langle  R_1(  \Theta_1 X^\e_0 +1_{[t,t+\e]} \dbV X^\e  ),   \Theta_1 X^\e_0 +1_{[t,t+\e]} \dbV X^\e  \rangle  \big] ds +  \langle G_1  X^\e_0 (T), X^\e_0 (T) \rangle  \]\\
&\q +  \mE_t \[ \int_{t}^{T} \big[ \langle Q_1 X,X^\e_0  \rangle + \langle  R_1 \Theta_1 X,   \Theta_1 X^\e_0 +1_{[t,t+\e]} \dbV X^\e  \rangle  \big] ds +  \langle G_1  X(T), X^\e_0 (T) \rangle  \]\\
&= \frac{1}{2}J_1+J_2+ \frac{1}{2}\mE_t  \int_{t}^{t+\e}  \langle  R_1  \dbV X^\e  ,    \dbV X^\e  \rangle   ds
+  \mE_t  \int_{t}^{t+\e}  \langle  R_1 \Theta_1 X,  \dbV X^\e  \rangle   ds+ \mE_t\int_{t}^{t+\e} \langle  \Theta_1^\top R_1\dbV X^\e, X^\e_0 \rangle ds,
\end{align*}
where $J_1$ and $J_2$ are given by 
\vspace{-1mm}
\begin{align*}
J_1&=\mE_t \[ \int_{t}^{T}  \big\langle \big(   Q_1+ \Theta_1^\top R_1 \Theta_1 \big) X^\e_0,X^\e_0  \big\rangle  ds +  \langle G_1  X^\e_0 (T), X^\e_0 (T) \rangle  \],\\
J_2&=  \mE_t \[ \int_{t}^{T}  \big\langle \big( Q_1 + \Theta_1^\top R_1 \Theta_1 \big) X,X^\e_0  \big\rangle  ds +  \langle G_1  X(T), X^\e_0 (T) \rangle  \].\vspace{-1mm}
\end{align*}

For $J_1$, we introduce the following adjoint equation:
\vspace{-1mm}
\begin{equation*}
\begin{cases}
dY_1=-\big[ \mathbf{A}^\top Y_1+\mathbf{C}^\top Z_1+ (Q_1+\Theta_1^\top R_1 \Theta_1) X^\e_0 \big]ds+ Z_1 dW(s), & \mbox{ in } [t,T],\\
Y_1(T)=G_1 X^\e_0(T).
\end{cases}
\end{equation*}
Applying It\^o's formula to $\langle Y_1,X^\e_0 \rangle$ gives
\vspace{-1mm}
\begin{align*}
J_1=\mE_t \int_{t}^{t+\e} \big\langle \big(  \mathbf{A}^\e -\mathbf{A}\big)^\top Y_1+  \big( \mathbf{C}^\e- \mathbf{C} \big)^\top  Z_1, X^\e \big\rangle ds.
\end{align*}
Set the ansatz
\vspace{-1mm}
\begin{align*}
\begin{cases}
Y_1=P_1 X^\e_0+ P_3,  \qq
dP_3=\Pi_3ds+V_3dW(s) & \mbox{ in }[t,T],\\
P_1(T)=G_1,\q P_3(T)=0,
\end{cases}
\end{align*}
with $P_1$ and  $P_3$ respectively solving 
\vspace{-1mm}
\begin{align*}
&\begin{cases}
\dot{P}_1=-\big( \mathbf{A}^\top P_1+ P_1 \mathbf{A} + \mathbf{C}^\top P_1 \mathbf{C} + Q_1+ \Theta_1^\top R_1 \Theta_1 \big), & \mbox{ in }[t,T],\\
P_1(T)=G_1,
\end{cases}
\\ \ns\ds
&\begin{cases}
dP_3=-\big[ \mathbf{A}^\top P_3 +\mathbf{C}^\top V_3+ \mathbf{C}^\top P_1 (\mathbf{C}^\e-\mathbf{C}) X^\e +  P_1(\mathbf{A}^\e-\mathbf{A}) X^\e \big]ds +V_3 dW(s), & \mbox{ in }[t,T],\\\ns\ds
P_2(T)=0.
\end{cases}
\end{align*}
Then
\begin{equation*}
Y_1=P_1X^\e_0+ P_3,\qq
Z_1=P_1\big[ \mathbf{C} X^\e_0+ (\mathbf{C}^\e-\mathbf{C}) X^\e \big] +V_3. 
\end{equation*}
Assuming $P_3 = \widetilde{P}_3 X^\varepsilon$ with $\widetilde{P}_3(T) = 0$, we obtain the equation for $\widetilde{P}_3$:
\vspace{-1mm}
\begin{align*}
\begin{cases}
\dot{\wt P}_3=-\big[ \mathbf{A}^\top \wt P_3 + \wt P_3 \mathbf{A}^\e + \mathbf{C}^\top \wt P_3 \mathbf{C}^\e + \mathbf{C}^\top P_1 (\mathbf{C}^\e-\mathbf{C})+ P_1 (\mathbf{A}^\e-\mathbf{A}) \big], & \mbox{ in }[t,T],\\
\wt P_3(T)=0,
\end{cases}
\end{align*}
and
$V_3=\wt P_3\mathbf{C}^\e X^\e$.
Consequently, 
\begin{equation*}
Y_1=P_1X^\e_0+\wt P_3X^\e,\qq
Z_1=P_1\big[ \mathbf{C} X^\e_0 +(\mathbf{C}^\e-\mathbf{C}) X^\e \big] + \wt P_3\mathbf{C}^\e X^\e. 
\end{equation*}
From estimates \eqref{pr-th1-est6}, \eqref{pr-th1-est3}, \eqref{pr-th1-est4}, boundedness of $P_1$, and Gr\"onwall's inequality, we obtain
\begin{equation*}
|\wt P_3(\cd)|_{C([t,T];\dbR^{n\times n})}\leq \cC(\Theta_1,\dbV) \e.
\end{equation*}
Then 
\vspace{-1mm}
\begin{equation}\label{J_1}
	\begin{aligned}
J_1&=\mE_t \int_{t}^{t+\e} \big\langle \big(  \mathbf{A}^\e-\mathbf{A} \big)^\top Y_1+  \big( \mathbf{C}^\e-\mathbf{C} \big)^\top  Z_1, X^\e \big\rangle ds \\
&= \mE_t \int_{t}^{t+\e} \!\!\Big \langle \big(  \mathbf{A}^\e\!-\!\mathbf{A} \big)^\top \big(P_1X^\e_0+\wt P_3X^\e\big)\!+\!  \big( \mathbf{C}^\e\!-\!\mathbf{C} \big)^\top \! \Big[P_1\big[ \mathbf{C} X^\e_0 \!+\!(\mathbf{C}^\e-\mathbf{C}) X^\e \big] \!+\! \wt P_3\mathbf{C}^\e X^\e\Big], X^\e \Big\rangle ds \\
&= \mE_t \int_{t}^{t+\e} \big\langle  ( \mathbf{C}^\e-\mathbf{C} )^\top  P_1 (\mathbf{C}^\e-\mathbf{C} )X^\e  , X^\e \big\rangle ds + 	 \cC(\Theta_1,\dbV)|\xi|^2o(\e). %\vspace{-2mm}
\end{aligned}
\end{equation}

For $J_2$, we introduce
\begin{equation*}
\begin{cases}
dY_2= -\big[ \mathbf{A}^\top Y_2 + \mathbf{C}^\top Z_2+ \big( Q_1+ \Theta_1^\top  R_1 \Theta_1\big)X \big]ds +Z_2 dW(s),\q \mbox{in } [t,T],\\
Y_2(T)= G_1 X(T).
\end{cases}
\end{equation*}
Applying It\^o's formula to $\langle Y_2,X^\e_0 \rangle $  gives
\begin{equation*}
\begin{aligned}
J_2= \mE_t \int_{t}^{t+\e} \big\langle  (\mathbf{A}^\e-\mathbf{A})^\top Y_2+ (\mathbf{C}^\e-\mathbf{C})^\top Z_2,X^\e \big\rangle ds.
\end{aligned}
\end{equation*}
Assuming $Y_2 = P_4 X$, we obtain the equation for $P_4$:
\vspace{-1mm}
\begin{equation*}
\begin{cases}
\dot{P}_4=-\big( \mathbf{A}^\top P_4+ P_4 \mathbf{A} + \mathbf{C}^\top P_4 \mathbf{C} + Q_1+ \Theta_1^\top R_1 \Theta_1 \big),\q \mbox{in } [t,T],\\
P_4(T)=G_1,
\end{cases}
\end{equation*}
and $Z_2 = P_4 \mathbf{C} X$. Clearly, $P_4(s) = P_1(s)$ for $s \in [t,T]$. Then we have 
\vspace{-1mm}
\begin{equation}\label{J_2}
	\begin{aligned}
J_2 &= \mE_t \int_{t}^{t+\e} \big\langle  (\mathbf{A}^\e-\mathbf{A})^\top Y_2+ (\mathbf{C}^\e-\mathbf{C})^\top Z_2,X^\e \big\rangle ds\\
&=  \mE_t \int_{t}^{t+\e} \big\langle  (\mathbf{A}^\e-\mathbf{A})^\top P_1 X+ (\mathbf{C}^\e-\mathbf{C})^\top P_1 \mathbf{C} X,X^\e \big\rangle ds.
\end{aligned}
\end{equation}

Utilizing estimates \eqref{pr-th1-est3}--\eqref{pr-th1-est4}, \eqref{J_1}--\eqref{J_2}, and the notations in \eqref{nota}, we obtain
\vspace{-1mm}
\begin{align*}
&\cJ_1(t,\xi; \Theta_1^\e X^\e,\Theta_2^*({\Theta}^\e_1) X^\e)- \cJ_1(t,\xi; \Theta_1 X,\Theta_2^*({\Theta}_1) X )\\
&=  \frac{1}{2}\mE_t \int_{t}^{t+\e}  \big\langle   \big[( \mathbf{C}^\e-\mathbf{C} )^\top P_1 (\mathbf{C}^\e-\mathbf{C}) + \dbV^\top R_1 \dbV \big]  X^\e,X^\e \big\rangle ds\\
&\q +  \mE_t \int_{t}^{t+\e} \big\langle  \big[  \dbV^\top R_1 \Theta_1+    (\mathbf{A}^\e-\mathbf{A})^\top P_1+ (\mathbf{C}^\e-\mathbf{C})^\top P_1 \mathbf{C}  \big] X ,X^\e \big\rangle ds+ \cC(\Theta_1,\dbV)|\xi|^2o(\e)  \\
&=   \frac{1}{2}\mE_t \int_{t}^{t+\e}  \big\langle    \dbV^\top \big[ \mathbf{D}(P_2)^\top P_1 \mathbf{D}(P_2)+ R_1\big] \dbV  X^\e,X^\e \big\rangle ds\\
&\q +  \mE_t \int_{t}^{t+\e} \big\langle   \dbV^\top  \big[   R_1 \Theta_1+    \mathbf{B}(P_2)^\top P_1+ \mathbf{D}(P_2)^\top P_1 \mathbf{C}  \big] X ,X^\e \big\rangle ds+ \cC(\Theta_1,\dbV)|\xi|^2o(\e)\\
&= \frac{1}{2} \int_{t}^{t+\e}  \big\langle    \dbV^\top \big[ \mathbf{D}(P_2)^\top P_1 \mathbf{D}(P_2)+ R_1\big] \dbV  \xi,\xi \big\rangle ds\\
&\q +    \int_{t}^{t+\e} \big\langle \dbV^\top  \big[  R_1 \Theta_1+    \mathbf{B}(P_2)^\top P_1+ \mathbf{D}(P_2)^\top P_1 \mathbf{C}  \big] \xi ,\xi \big\rangle ds+
\cC(\Theta_1,\dbV)|\xi|^2o(\e),
\end{align*}
where the last equality follows from
\vspace{-1mm}
\begin{equation*}
\mE_t \Big[ \sup_{s\in [t,t+\e]} |X(s)-\xi|^2  \Big]\leq \cC (\Theta_1) |\xi|^2\e, \quad
\mE_t \Big[ \sup_{s\in [t,t+\e]} |X^\e(s)-\xi|^2  \Big]\leq \cC (\Theta_1,\dbV) |\xi|^2\e.
\end{equation*}
Consequently, Definition \ref{De-closedloop} implies
\vspace{-1mm}
\begin{align*}
&\liminf_{\e\to 0+} \frac{\cJ_1\big(t,\xi;\bar{\Theta}^\e_1 X^\e,\Theta_2^*(\bar{\Theta}^\e_1) X^\e\big)-\cJ_1\big(t,\xi;\bar{\Theta}_1 \bar{X},\Theta_2^*(\bar{\Theta}_1) \bar{X} \big)}{\e}\\
&= \frac{1}{2} \big\langle  \big( \mathbf{D}(P_2)^\top P_1 \mathbf{D}(P_2)+ R_1\big) \dbV  \xi,\dbV \xi \big\rangle+ \big\langle   \big( R_1 \Theta_1+    \mathbf{B}(P_2)^\top P_1+ \mathbf{D}(P_2)^\top P_1 \mathbf{C}  \big) \xi, \dbV \xi \big\rangle \geq 0,  \, \as
\end{align*}
From the arbitrariness of $t,\xi,\dbV$, we know that $\Theta_1$ is a closed-loop equilibrium strategy if and only if
\begin{enumerate}
\item[(i)] $\mathbf{D}(P_2)^\top P_1 \mathbf{D}(P_2)+ R_1\geq 0$;
\item[(ii)] $ R_1 \Theta_1+    \mathbf{B}(P_2)^\top P_1+ \mathbf{D}(P_2)^\top P_1 \mathbf{C}=0$.
\end{enumerate}
However, by Assumption \ref{H2} and properties of Lyapunov differential equations \cite[Lemma 7.3]{Yong-Zhou1999}, we obtain $\mathbf{D}(P_2)^\top P_1 \mathbf{D}(P_2) + R_1 \geq \delta I$. Hence, condition (i) holds automatically. Thus we complete the proof.
\end{proof}

\vspace{-3mm}

\subsection{Proof of Theorem \ref{th2}}\label{Appendix-pr-th2}

\begin{proof}
We divide the proof into three steps.

{\bf Step 1.} In this step, we introduce Hamiltonians for the leader and follower and show local minimum property for $u^*$ and $v^*$ (e.g., \eqref{pr-th2-eq1},
\eqref{pr-th2-eq2}).

For $(t,x,\mu,\nu,p_1,q_1),(t,x,\mu,\nu,p_2,q_2)\in[0,T]\times \dbR^n\times \dbR^{m_1}\times \dbR^{m_2}\times \dbR^n\times\dbR^{n\times n}$, define
\vspace{-2mm}
\begin{equation*}
\begin{aligned}
H_1(t,x,\mu ,\nu ,p_1,q_1)&=p_1^\top \big[A(t) x+B_1(t)\mu+B_2 \nu\big]+ \frac{1}{2} x^\top Q_1(t) x+ \frac{1}{2} \mu^\top R_1(t) \mu \\&\q+ \frac{1}{2}\mathrm{Tr}\Big\{q_1 \big[C(t)x+D_1(t)\mu+D_2(t)\nu\big]\big[C(t)x+D_1(t)\mu+D_2(t)\nu\big]^\top \Big\},\\
H_2(t,x,\mu,\nu,p_2,q_2)&=p_2^\top \big[A(t) x+B_1(t)\mu+B_2 \nu\big]+ \frac{1}{2} x^\top Q_2(t) x+ \frac{1}{2} \nu^\top R_2(t) \nu \\&\q+ \frac{1}{2}\mathrm{Tr}\Big\{q_2 \big[C(t)x+D_1(t)\mu+D_2(t)\nu\big]\big[C(t)x+D_1(t)\mu+D_2(t)\nu\big]^\top \Big\},
\end{aligned}
\end{equation*}
where $H_1$ and $H_2$ are the Hamiltonians for the leader and follower, respectively.

Fix $(t,x,\mu)\in [0,T]\times \dbR^n\times \dbR^{m_1}$. Substituting $p_2=P_2(t)x$, $q_2=P_2(t)$ into $H_2$ and minimizing over $\nu$ yields the unique minimizer
\vspace{-2mm}
\begin{align*}
&\mathrm{arg} \min_{\nu\in\dbR^{m_2}} H_2\big(t,x,\mu,\nu,P_2(t)x,P_2(t)\big)
\\&=  - \big[ \big( R_2+ D_2^{\top} P_2 D_2 \big)^{-1} \big( B_2^\top P_2+ D_2^\top P_2 C \big)\big](t)\,x -\big[\big( R_2+ D_2^{\top} P_2 D_2 \big)^{-1}  D_2^\top P_2 D_1 \big](t) \,\mu.
\end{align*}
By \eqref{v*}, $v^*(t,x,\mu)$ equals this minimizer, so 
\vspace{-1mm}
\begin{align}\label{pr-th2-eq1}
H_2\big(t,x,\mu,v^*(t,x,\mu),P_2(t)x,P_2(t)\big) \leq H_2\big(t,x,\mu,\nu,P_2(t)x,P_2(t)\big),\q \forall \nu\in \dbR^{m_2}.
\end{align}

Next, fix $(t,x) \in [0,T] \times \mathbb{R}^n$. Substituting $p_1 = P_1(t)x$, $q_1 = P_1(t)$, and $v = v^*(t,x,\mu)$ into $H_1$ and minimizing over $\mu$ gives 
\vspace{-2mm}
\begin{align*}
&\mathrm{arg} \min_{\mu\in\dbR^{m_1}} H_1\big(t,x,\mu,v^*(t,x,\mu),P_1(t)x,P_1(t)\big)
\\&= -\Big(\big[R_1\!+\! \mathbf{D}(P_2)^\top P_1 \mathbf{D}(P_2)  \big]^{-1} \big\{  \mathbf{B}(P_2)^\top  P_1\!+\! \mathbf{D}(P_2)^\top P_1  \big[ C \! -\!D_2 \big( R_2\!+\! D_2^{\top} P_2 D_2 \big)^{-1} \big( B_2^\top P_2\!+\! D_2^\top P_2 C \big)\big]   \big\} \Big) (t) x.\vspace{-1mm}
\end{align*}
From \eqref{bar-Theta1} and \eqref{u*}, $u^*(t,x)$ coincides with this minimizer, hence
\vspace{-1mm}
\begin{align}\label{pr-th2-eq2}
H_1\big(t,x,u^*(t,x),v^*(t,x,u^*(t,x)),P_1(t)x,P_1(t)\big) \leq H_1\big(t,x,\mu,v^*(t,x,\mu),P_1(t)x,P_1(t)\big), \q \forall \mu \in \mathbb{R}^{m_1}.
\end{align}

{\bf Step 2.} In this step, we prove that $(u^*,v^*)$ is a feedback Stackelberg equilibrium as defined in \eqref{De-feed-Sta}.

Fix the leader's strategy $u^*(s,x)$ and let the follower choose any $v\in\mathcal{V}$ (defined in \eqref{feed-Sta-control}). For any initial pair $(t,\xi) \in [0,T] \times \mathbb{R}^n$, denote by $X(\cdot)$ the corresponding state trajectory. Applying It\^o's formula to $\frac12 X(s)^\top P_2(s)X(s)$ over $[t,T]$ gives 
\vspace{-2mm}
\begin{align*}
&\frac{1}{2}\mE_t \big[ X(T)^\top G_2 X(T)\big]-\frac{1}{2} \mE_t\big[ \xi^\top P_2(t)\xi \big]\\
&= \mE_t \int_{t}^{T} \Big\{ \frac{1}{2} X^\top \dot{P}_2 X+ X^\top P_2 \big[AX+ B_1u^* +B_2 v(s,X,u^*) \big] \\
&\qq\q + \frac{1}{2}\mathrm{Tr}\big[ P_2\big(CX+D_1 u^*+D_2 v(s,X,u^*) \big) \big( CX+D_1 u^*+D_2 v(s,X,u^*) \big)^\top \big]  \Big\}ds.
\end{align*}
Using \eqref{pr-th2-eq1} and the $P_2$-equation in \eqref{ERE-2}, we obtain 
%
%\vspace{-2mm}
\begin{align*}
&\frac{1}{2} \mE_t\big[ \xi^\top P_2(t)\xi \big]\\ & =\cJ_2(t,\xi;u^*,v(u^*)) -\mE_t \int_{t}^{T} \Big[\frac{1}{2} X^\top \dot{P}_2 X + H_2\big(s,X,u^*(s,X),v(s,X,u^*(s,X)),P_2X,P_2 \big) \Big] ds\\
&\leq \cJ_2(t,\xi;u^*,v(u^*))-\mE_t \int_{t}^{T} \Big[\frac{1}{2} X^\top \dot{P}_2 X + H_2\big(s,X,u^*(s,X),v^*(s,X,u^*(s,X)),P_2X,P_2 \big) \Big] ds\\
&=\cJ_2(t,\xi;u^*,v(u^*))-\mE_t \int_{t}^{T} \Big[\frac{1}{2} X^\top \dot{P}_2 X + H_2\big(s,X,\bar{\Theta}_1X,\Theta_2^*(\bar{\Theta}_1)X,P_2X,P_2 \big) \Big] ds\\
&=\cJ_2(t,\xi;u^*,v(u^*)).%\vspace{-1mm}
\end{align*}
Equality holds when $v=v^*$, and then $\mathcal{J}_2(t,\xi;u^*,v^*(u^*)) = \tfrac12\mathbb{E}_t[\xi^\top P_2(t)\xi]$. Thus, 
%
%\vspace{-1mm}
\begin{equation*}
\cJ_2(t,\xi;u^*,v^*(u^*))\leq \cJ_2(t,\xi;u^*,v(u^*)), \q \forall v\in \cV.
\end{equation*}
Similarly, we can show
\vspace{-1mm}
\begin{align*}
\cJ_1(t,x;u^*,v^*(u^*))\leq \cJ_1(t,x;u,v^*(u)),\q \forall u\in \cU.
\end{align*}
This completes the proof.
\end{proof}

\vspace{-3mm}

\subsection{Proof of Theorem \ref{th3}}\label{Appendix-pr-th3}

\begin{proof}
We prove the result using Lemma \ref{lm2.3}. Rewrite the ERE \eqref{ERE-2}--\eqref{Th1-Th2} in column form: 
\vspace{-1mm}
\begin{equation}
x(t)=\begin{pmatrix}
\operatorname{vec}(P_1(t))\\\operatorname{vec}(P_2(t))
\end{pmatrix},\q \dot{x}(t)=\begin{pmatrix}
f_1(t,x(t))\\f_2(t,x(t))
\end{pmatrix}\vspace{-1mm}
\end{equation}
and choose the domain $\cD = [0,T] \times \{ x \in \mathbb{R}^{2n^2} \mid |x| \leq 3 \mathcal{C}_{\mathrm{pri}} \}$. The proof is divided into two steps.

{\bf Step 1.} In this step, we establish Lipschitz estimates for $f_1$ and $f_2$ over the domain $\cD$.

For any $P_i,P_i'\in\mathbb{S}_+^n$ ($i=1,2$) satisfying 
\vspace{-1mm}
$$
\max\left\{\left |\begin{pmatrix} \operatorname{vec}(P_1)\\\operatorname{vec}(P_2) \end{pmatrix}\right| , \left |\begin{pmatrix} \operatorname{vec}(P_1')\\\operatorname{vec}(P_2') \end{pmatrix}\right| \right\}\leq  3 \mathcal{C}_{\mathrm{pri}},\vspace{-1mm}
$$
direct computation yields
\vspace{-1mm}
\begin{equation}\label{pr-th3-eq1}
\begin{cases}
\ds|\mathbf{B}(P_2)-\mathbf{B}(P'_2)|\leq  \mathcal{C}(\mathcal{C}_{\mathrm{pri}}) |P_2-P_2'|\leq \mathcal{C}(\mathcal{C}_{\mathrm{pri}}) | \operatorname{vec}(P_2)-\operatorname{vec}(P_2')|,\\
\ns\ds |\mathbf{D}(P_2)-\mathbf{D}(P'_2)|\leq  \cC(\mathcal{C}_{\mathrm{pri}}) |P_2-P_2'|\leq \mathcal{C}(\mathcal{C}_{\mathrm{pri}}) |\operatorname{vec}(P_2)-\operatorname{vec}(P_2')|,
\end{cases}\vspace{-1mm}
\end{equation}
where $\mathcal{C}(\mathcal{C}_{\mathrm{pri}})$ is a generic constant depending on $\mathcal{C}_{\mathrm{pri}}$ and system parameters.
Define
\vspace{-1mm}
\begin{align*}
\!\!\begin{cases}
\bar{\Theta}_1\!=\!-\big[R_1\!+\! \mathbf{D}(P_2)^\top P_1 \mathbf{D}(P_2)  \big]^{-1}\! \big\{  \mathbf{B}(P_2)^\top  P_1\!+\! \mathbf{D}(P_2)^\top P_1  \big[ C \! -\!D_2 \big( R_2\!+\! D_2^{\top}\! P_2 D_2 \big)^{-1} \big( B_2^\top P_2\!+\! D_2^\top P_2 C \big)\big]   \big\},\\
\ns\ds\bar{\Theta}_1'\!=\!-\big[R_1\!+\! \mathbf{D}(P_2')^\top P_1' \mathbf{D}(P_2')  \big]^{-1}\! \big\{  \mathbf{B}(P_2')^\top  P_1'\!+\! \mathbf{D}(P_2')^\top P_1'  \big[ C\!-\!D_2 \big( R_2\!+\! D_2^{\top}\! P_2' D_2 \big)^{-1} \big( B_2^\top P_2'\!+\! D_2^\top P_2' C \big)\big]   \big\}.
\end{cases}\vspace{-1mm}
\end{align*}
From \eqref{pr-th3-eq1}, we obtain
\vspace{-1mm}
\begin{equation}\label{pr-th3-eq2}
\begin{aligned}
|\bar{\Theta}_1-\bar{\Theta}_1' |\leq \cC(\mathcal{C}_{\mathrm{pri}}) \big(|\operatorname{vec}(P_1)-\operatorname{vec}(P_1')|+|\operatorname{vec}(P_2)-\operatorname{vec}(P_2')| \big),
\end{aligned}\vspace{-1mm}
\end{equation}
which together with \eqref{Th1-Th2} implies
\begin{equation}\label{pr-th3-eq3}
\begin{aligned}
|\Theta_2^*(\bar{\Theta}_1)-\Theta_2^*(\bar{\Theta}_1' )|\leq \cC(\mathcal{C}_{\mathrm{pri}}) \big(|\operatorname{vec}(P_1)-\operatorname{vec}(P_1')|+|\operatorname{vec}(P_2)-\operatorname{vec}(P_2')| \big).
\end{aligned}
\end{equation}
Set $x={\small\begin{pmatrix}
\operatorname{vec}(P_1)\\\operatorname{vec}(P_2)
\end{pmatrix}}$ and $y={\small\begin{pmatrix}
\operatorname{vec}(P_1')\\\operatorname{vec}(P_2')
\end{pmatrix}}$. From estimates \eqref{pr-th3-eq2}--\eqref{pr-th3-eq3}, direct computation gives
\begin{equation}\label{pr-th3-eq4}
\begin{aligned}
|f_i(t,x)\!-\!f_i(t,y)|\leq \cC(\mathcal{C}_{\mathrm{pri}}) \big(|\operatorname{vec}(P_1)\!-\!\operatorname{vec}(P_1')|+|\operatorname{vec}(P_2)\!-\!\operatorname{vec}(P_2')| \big)\leq \cC(\mathcal{C}_{\mathrm{pri}}) |x-y|,\, (i=1,2).
\end{aligned}
\end{equation}

{\bf Step 2.} In this step, we apply Lemma~\ref{lm2.3} to establish the global well-posedness of the ERE \eqref{ERE-2}--\eqref{Th1-Th2}.

With the Lipschitz estimate \eqref{pr-th3-eq4}, all conditions of Lemma~\ref{lm2.3} are satisfied. Hence, there exists $\delta>0$ such that the ERE \eqref{ERE-2}--\eqref{Th1-Th2} admits a unique solution on $[T-\delta,T]$, and this solution can be continued until it reaches the boundary  
%
%\vspace{-1mm}
\begin{equation}
\begin{aligned}
\partial \cD=\bigl(\{0,T\}\times \big\{x\in\mathbb{R}^{2n^2}\mid |x|\le 3\cC_{\mathrm{pri}}\big\}\bigr)\,\cup\,\bigl((0,T)\times \big\{x\in\mathbb{R}^{2n^2}\mid |x|=3\cC_{\mathrm{pri}}\big\}\bigr).
\end{aligned}%\vspace{-1mm}
\end{equation}
Since the estimate \eqref{th3-eq1} gives $\sup_{t\in[0,T]}|x(t)|\leq \mathcal{C}_{\mathrm{pri}}$, the solution cannot reach the spatial boundary $(0,T) \times \{ x \in \mathbb{R}^{2n^2} \mid |x| = 3 \mathcal{C}_{\mathrm{pri}} \}$; the only possible boundary it may hit is $\{0,T\} \times \{ x \in \mathbb{R}^{2n^2} \mid |x| \leq 3 \mathcal{C}_{\mathrm{pri}} \}$. Consequently, the solution can be uniquely extended to the entire interval $[0,T]$. 
\end{proof}

\vspace{-3mm}

\subsection{Proof of Theorem \ref{th4}}\label{Appendix-pr-th4}

To prove Theorem \ref{th4}, we first establish a technical lemma.
\begin{lemma}\label{lm5.1}
Let $B_2(\cdot)\in C^1([0,T];\mathbb R^{n\times m_2})$ and $R_2(\cdot)\in C^1([0,T];\mathbb R^{m_2\times m_2})$.
Assume there exist constants $\delta>0$ and $v>0$ such that for all $s\in[0,T]$
\vspace{-2mm}
$$
R_2(s)\geq \delta I_{m_2},\qquad S(s):=B_2(s)R_2(s)^{-1}B_2(s)^\top \geq v I_n.\vspace{-1mm}
$$
Define
$R(s):=S(s)^{1/2}$ and $T(s):=R(s)^{-1}$.
Then $S(\cd),R(\cd),T(\cd) \in C^1([0,T];\mathbb R^{n\times n})$, with
\vspace{-2mm}
\begin{equation*}
|R'(s)|\le \frac{|\operatorname{vec}(S(s)')|}{2\sqrt{v}},\q |T'(s)|\le \frac{|\operatorname{vec}(S(s)')|}{2 v^{3/2}}.\vspace{-1mm}
\end{equation*}
\end{lemma}
\begin{proof}
Since $R_2(s)\geq \delta I_{m_2}>0$ and $R_2(\cdot)\in C^1([0,T];\mathbb R^{m_2\times m_2})$, we have $R_2(\cd)^{-1}\in C^1([0,T];\mathbb R^{m_2\times m_2})$ and
\vspace{-2mm}
$$
\frac{d}{ds}\big(R_2(s)^{-1}\big)=-R_2(s)^{-1}R_2'(s)R_2(s)^{-1},\vspace{-1mm}
$$
which together with $B_2(\cdot)\in C^1([0,T];\mathbb R^{n\times m_2})$ implies $S(\cd)\in C^1([0,T];\mathbb R^{n\times n})$.
From $S(s)\geq v I_n$, we know that $R(s)=S(s)^{1/2}$ and $T(s)= S(s)^{-1}=S(s)^{-1/2}$ are well defined.

We now prove that $R(\cd)\in C^1([0,T];\mathbb R^{n\times n})$.
Consider  $F: \mathbb{R}^{n \times n} \times \mathbb{R}^{n \times n} \to \mathbb{R}^{n \times n}$ defined by $
F(X, Y) = X^2 - Y$,
whose partial Fr\'echet derivative with respect to $X$ at  $(X, Y)$ is given by
\vspace{-2mm}
\begin{equation*}
D_X F(X, Y)[H] = \lim_{\epsilon \to 0} \frac{(X+\epsilon H)^2 - X^2}{\epsilon} = X H + H X.\vspace{-1mm}
\end{equation*}
Clearly, for any $s\in [0,T]$, we have $F(R(s),S(s))=0$ and $D_X F(R(s), S(s))[H]=R(s)H+HR(s)$.
Next, we show that the mapping $D_X F(R(s), S(s)): H \mapsto R(s)H + H R(s)$ is an isomorphism. First, it is clear that $D_X F(R(s), S(s)): \mathbb{R}^{n \times n} \to \mathbb{R}^{n \times n}$ is linear. Second, consider the following Sylvester equation: 
\vspace{-1mm}
\begin{equation}\label{pr-lm2.5-eq1}
R(s)H+HR(s)= Z.\vspace{-1mm}
\end{equation}
By Lemma \ref{lm2.1}, equation \eqref{pr-lm2.5-eq1} is equivalent to
\vspace{-1mm}
\begin{equation}
\big[I_n \otimes R(s) + R(s) \otimes I_n\big] \operatorname{vec}(H) = \operatorname{vec}(Z).\vspace{-1mm}
\end{equation}
From Lemma \ref{lm2.2}, we know that $I_n \otimes R(s) + R(s) \otimes I_n$ is symmetric and satisfies $I_n \otimes R(s) + R(s) \otimes I_n \geq 2\sqrt{\nu} I_{n^2}$. Hence, the mapping $D_X F(R(s), S(s))$ is a linear bijection on a finite-dimensional space and therefore an isomorphism. Consequently, Lemma \ref{lm2.4} implies that the mapping $Y \mapsto Y^{1/2}$ is continuously differentiable in a neighborhood of $S(s)$ for any $s \in [0,T]$. This, together with $S(\cdot) \in C^1([0,T]; \mathbb{R}^{n \times n})$, implies that the  function $R(s) = (S(s))^{1/2}$ is continuously differentiable on $[0,T]$.

Differentiating $R(s)^2=S(s)$ gives $R(s)'R(s)+R(s)R(s)'=S(s)'$, which via Lemma \ref{lm2.1} becomes
\begin{equation}
\big[I_n \otimes R(s) + R(s) \otimes I_n\big] \operatorname{vec}(R(s)') = \operatorname{vec}(S(s)').
\end{equation}
Hence
\vspace{-1mm}
\begin{align}\label{pr-lm5.1-eq2}
\big|\operatorname{vec}(R(s)')\big|=\Big|\big[ I_n \otimes R(s) + R(s) \otimes I_n\big]^{-1}\operatorname{vec}(S(s)') \Big|\leq \frac{|\operatorname{vec}(S(s)')|}{2\sqrt{v}}.
\end{align}
Since $R(s)\geq\sqrt{\nu}I_n$ and $T(s)=R(s)^{-1}$, $T(\cdot)$ is $ C^1$ and $T(s)' = -R(s)^{-1}R(s)'R(s)^{-1}$, so 
\vspace{-1mm}
\begin{equation*}
|T(s)'|\le \big|R(s)^{-1}\big|^2 |R(s)'|\le \frac{|\operatorname{vec}(S(s)')|}{2 v^{3/2}}.
\end{equation*}
This completes the proof.
\end{proof}

Now we are in a position to prove Theorem \ref{th4}.
\begin{proof}
We divide the proof into four steps.

{\bf Step 1.} In this step, we derive a preliminary estimate.

From \eqref{Th1-Th2}, we have
\vspace{-1mm}
\begin{equation}\label{pr-th4-eq1}
\begin{aligned}
&  2\bar{\Theta}_1^\top \big[ R_1+ \mathbf{D}(P_2)^\top P_1 \mathbf{D}(P_2)  \big]\bar{\Theta}_1 +  \bar{\Theta}_1^\top \mathbf{B}(P_2)^\top P_1 + P_1 \mathbf{B}(P_2) \bar{\Theta}_1 \\
& + \bar{\Theta}_1^\top \mathbf{D}(P_2)^\top P_1 \big[ C- D_2 (R_2+D_2^\top P_2 D_2)^{-1} (B_2^\top P_2+ D_2^\top P_2 C)  \big] \\
&+   \big[ C- D_2 (R_2+D_2^\top P_2 D_2)^{-1} (B_2^\top P_2+ D_2^\top P_2 C)  \big]^\top P_1 \mathbf{D}(P_2) \bar{\Theta}_1=0.
\end{aligned}\vspace{-1mm}
\end{equation}
Substituting \eqref{nota} and \eqref{pr-th4-eq1} into the $P_1$-equation yields
\vspace{-1mm}
\begin{equation}\label{pr-th4-eq7}
	\begin{aligned}
\dot{P}_1&=\!-\big\{\big[  A^\top\!\! -\! \big(P_2B_2 \!+\! C^\top P_2 D_2\big) \big(R_2\!+\! D_2^\top P_2 D_2\big)^{-1}\! B_2^\top  \big] P_1\! \\&\qq 
+\! P_1 \big[  A\!-\! B_2\big(R_2\!+\!D_2^\top\! P_2 D_2\big)^{-1} \big(B_2^\top P_2 \!+\!D_2^\top P_2 C\big) \big] \\
&\qq  +  \big[ C\!-\! D_2\big(R_2\!+\!D_2^\top P_2 D_2\big)^{-1}\! \big(B_2^\top P_2\! +\!D_2^\top P_2 C\big) \big]^\top\! \\&
\qq\times P_1\big[ C\!-\! D_2\big(R_2\!+\!D_2^\top P_2 D_2\big)^{-1}\! \big(B_2^\top P_2 \!+\!D_2^\top P_2 C\big) \big]\\
&\qq  +Q_1  - \bar{\Theta}_1^\top \big[R_1+ \mathbf{D}(P_2)^\top P_1 \mathbf{D}(P_2) \big] \bar{\Theta}_1  \big\}.
\end{aligned}\vspace{-1mm}
\end{equation}
By Assumption \ref{H2} and standard properties of Lyapunov differential equations (see, e.g., \cite[Lemma 7.3]{Yong-Zhou1999}), if the ERE \eqref{ERE-2}--\eqref{Th1-Th2} admits a solution $(P_1, P_2)$, then it necessarily holds that 
\vspace{-1mm}
\begin{equation}\label{pr-th4-eq2}
P_1(s)\geq 0,\q P_2(s)\geq 0, \q \forall s\in [0,T].\vspace{-1mm}
\end{equation}
From \eqref{pr-th4-eq2} and Assumption \ref{H2}, we have $\bar{\Theta}_1^\top \big[R_1+ \mathbf{D}(P_2)^\top P_1 \mathbf{D}(P_2) \big] \bar{\Theta}_1 \geq \delta I$,
which leads to 
\vspace{-1mm}
\begin{equation}\label{pr-th4-eq3}
	\begin{aligned}
\dot{P}_1& \!\geq \!-\big\{  \big[  A^\top \!-\! \big(P_2B_2 \!+\! C^\top P_2 D_2\big) \big(R_2\!+\! D_2^\top P_2 D_2\big)^{-1}\! B_2^\top\!  \big] P_1 \\&\qq
+ P_1 \big[  A\!-\! B_2\big(R_2\!+\!D_2^\top\! P_2 D_2\big)^{-1}\! \big(B_2^\top P_2 \!+\!D_2^\top\! P_2 C\big) \big] \\
&\qq  + \! \big[ C\!-\! D_2\big(R_2\!+\!D_2^\top P_2 D_2\big)^{-1} \big(B_2^\top P_2\! +\!D_2^\top\! P_2 C\big) \big]^\top  \\ &\qq
\times P_1\big[ C\!-\! D_2\big(R_2\!+\!D_2^\top\! P_2 D_2\big)^{-1}\! \big(B_2^\top P_2 \!+\!D_2^\top P_2 C\big) \big] \!+\! Q_1 \big\}.
\end{aligned}\vspace{-1mm}
\end{equation}

{\bf Step 2.} In this step, we establish a priori estimates for {\bf Case (i)}.

In the one-dimensional case with $|D_2|\geq \delta>0$, using Assumption \ref{H2}, we can directly obtain
\vspace{-1mm}
\begin{align}\label{pr-th4-eq4}
\big|(R_2+D_2^\top P_2 D_2)^{-1} (B_2^\top P_2+ D_2^\top P_2 C)\big| \leq \max\left\{ \frac{|B_2|+|D_2||C|}{\delta},\frac{|B_2|+|D_2||C|}{\delta^2} \right\},\q \forall P_2\geq 0.\vspace{-1mm}
\end{align}
From \eqref{pr-th4-eq3},  we have
\vspace{-1mm}
\begin{align*}
P_1(t)&\leq G_1+\int_{t}^{T} \Big\{  \big[  A^\top - \big(P_2B_2 + C^\top P_2 D_2\big) \big(R_2+ D_2^\top P_2 D_2\big)^{-1} B_2^\top  \big] P_1\\
&\qq \qq \qq+ P_1 \big[  A- B_2\big(R_2+D_2^\top P_2 D_2\big)^{-1} \big(B_2^\top P_2 +D_2^\top P_2 C\big) \big] \\
&\qq \qq \qq+  \big[ C- D_2\big(R_2+D_2^\top P_2 D_2\big)^{-1} \big(B_2^\top P_2 +D_2^\top P_2 C\big) \big]^\top  \\
& \qq \qq\qq \times  P_1 \big[ C- D_2\big(R_2+D_2^\top P_2 D_2\big)^{-1} \big(B_2^\top P_2 +D_2^\top P_2 C\big) \big] +Q_1 \Big\}ds.\vspace{-1mm}
\end{align*}
This, together with \eqref{pr-th4-eq4}, implies
\vspace{-1mm}
\begin{equation*}
|P_1(t)|\leq \cC + \int_{t}^{T} \big[\cC |P_1(s)| + \cC \big]ds,\vspace{-1mm}
\end{equation*}
where $\mathcal{C}$ is a generic constant depending only on the system parameters. By Gr\"onwall's inequality, we obtain the boundedness of $P_1$.

For $P_2$, note that
\vspace{-1mm}
\begin{align*}
\dot{P}_2&=- \big( A+B_1 \bar{\Theta}_1\big)^\top P_2 - P_2(A+B_1 \bar{\Theta}_1)- \big(C+D_1 \bar{\Theta}_1\big)^\top P_2 \big(C+D_1 \bar{\Theta}_1\big) - Q_2\\\ns\ds
&\q +\big[P_2B_2+ (C+D_1 \bar{\Theta}_1)^\top P_2 D_2 \big] \big[R_2+ D_2^\top P_2 D_2\big]^{-1} \big[B_2^\top P_2+ D_2^\top P_2 (C+D_1\bar{\Theta}_1) \big],
\end{align*}
where
\vspace{-1mm}
\begin{align*}
\bar{\Theta}_1&=-\big[R_1\!+\! \mathbf{D}(P_2)^\top P_1 \mathbf{D}(P_2)  \big]^{-1} \big\{  \mathbf{B}(P_2)^\top  P_1\!+\! \mathbf{D}(P_2)^\top P_1  \big[ C \! -\!D_2 \big( R_2\!+\! D_2^{\top}\! P_2 D_2 \big)^{-1} \big( B_2^\top P_2\!+\! D_2^\top P_2 C \big)\big]   \big\},\vspace{-1mm}
\end{align*}
and
\vspace{-1mm}
\begin{align*}
& \mathbf{B}(P_2):= B_1 -B_2(R_2+ D_2^\top P_2 D_2)^{-1} D_2^\top P_2 D_1,\qq \mathbf{D}(P_2):= D_1 -D_2(R_2+ D_2^\top P_2 D_2)^{-1} D_2^\top P_2 D_1.
\end{align*}
Similar to \eqref{pr-th4-eq4}, we can directly obtain
\vspace{-2mm}
\begin{equation} \label{pr-th4-eq5}
\begin{aligned}
&|\mathbf{B}(P_2)|\leq |B_1|+ |B_2| \max\left \{\frac{|D_1||D_2|}{\delta},\frac{|D_1||D_1|}{\delta^2} \right\},\q \forall P_2\geq 0,\\
&|\mathbf{D}(P_2)|\leq |D_1|+ |D_2| \max\left \{\frac{|D_1||D_2|}{\delta},\frac{|D_1||D_1|}{\delta^2} \right\},\q \forall P_2\geq 0.
\end{aligned}\vspace{-1mm}
\end{equation}
From \eqref{pr-th4-eq4}, \eqref{pr-th4-eq5} and the boundedness of $P_1$, we obtain the boundedness of $\bar{\Theta}_1$.
Since $P_2$ satisfies
\vspace{-1mm}
\begin{align*}
\dot{P}_2 \geq- \big( A+B_1 \bar{\Theta}_1\big)^\top P_2 - P_2(A+B_1 \bar{\Theta}_1)- (C+D_1 \bar{\Theta}_1)^\top P_2 (C+D_1 \bar{\Theta}_1) - Q_2,
\end{align*}
we obtain the boundedness of $P_2$ from the boundedness of $\bar{\Theta}_1$ and Gr\"onwall's inequality.

\ss

{\bf Step 3.}  In this step, we establish a priori estimates for {\bf Case (ii)}.

In the one-dimensional case with $D_2 = 0$, we have
\vspace{-1mm}
\begin{equation*}
\mathbf{B}(P_2)=B_1,\q \mathbf{D}(P_2)=D_1, \q \bar{\Theta}_1=-\big( R_1+ D_1^\top P_1 D_1 \big)^{-1} \big( B_1^\top P_1 + D_1^\top P_1 C \big).\label{pr-th4-eq6}\vspace{-1mm}
\end{equation*}
Moreover, estimate \eqref{pr-th4-eq3} reduces to
\vspace{-1mm}
\begin{equation*}
\dot{P}_1 \geq -\big[ \big(  A^\top - P_2B_2  R_2^{-1} B_2^\top  \big) P_1 + P_1 \big(  A - B_2R_2^{-1}  B_2^\top P_2  \big)   +   C ^\top P_1  C  +Q_1 \big].\vspace{-1mm}
\end{equation*}
Since $
P_i(\cd)\geq 0$ for $i=1,2$, and
$B_2R_2^{-1}B_2^\top \geq 0$,
we have
\vspace{-1mm}
\begin{equation*}
\dot{P}_1 \geq -\big(  A^\top  P_1 + P_1  A      +   C ^\top P_1  C  +Q_1 \big).\vspace{-1mm}
\end{equation*}
By Gr\"onwall's inequality, we obtain the boundedness of $P_1$, which also leads to the boundedness of $\bar{\Theta}_1$. Following the same procedure as in Step 2, we obtain the boundedness of $P_2$.

{\bf Step 4.} We establish a priori estimates for {\bf Case (iii)}.

In the high-dimensional case with $D_2 = 0$, the equation for $P_1$ (e.g., \eqref{pr-th4-eq7}) reduces to 
\vspace{-1mm}
\begin{equation*}
\dot{P}_1  = -\big[ \big( A^\top - P_2 B_2 R_2^{-1} B_2^\top \big) P_1 + P_1 \big( A - B_2 R_2^{-1} B_2^\top P_2 \big) + C^\top P_1 C + Q_1 - \bar{\Theta}_1^\top \big(R_2+ D_1^\top P_1 D_1 \big)\bar{\Theta}_1 \big].\vspace{-1mm}
\end{equation*}
Set $S(s)=B_2(s)R_2(s)^{-1}B_2(s)^\top$. Since $R_2(\cdot)\geq\delta I$ and $\operatorname{Rank}(B_2(t))=n$, $S(s)$ is positive definite. By continuity, $\exists\nu>0$ such that $S(s)\geq\nu I$. Lemma \ref{lm5.1} gives $S^{1/2},S^{-1/2}$ is $ C^1$. The equation for $P_1$ becomes
\vspace{-1mm}
\begin{align}
\dot{P}_1 &= P_2S P_1+ P_1 S P_2 -\big[  A^\top P_1 + P_1  A  + C^\top P_1 C + Q_1 - \bar{\Theta}_1^\top \big(R_2+ D_1^\top P_1 D_1 \big)\bar{\Theta}_1 \big].\vspace{-1mm}
\end{align}
Now we divide the proof of this part into {\bf Step (i)} to {\bf Step (iii)}.

\ss

{\bf Step (i).}   Let $\Phi(t,s)$ solve $\partial_s\Phi(s,t) = -S(s)P_2(s)\Phi(s,t)$ with $\Phi(t,t)=I$. By variation of constants formula,
\vspace{-1mm}
\begin{eqnarray}\label{pr-th4-eq11}
P_1(t) \! = \!\Phi(T, t)^\top\! G_1 \Phi(T, t) \!+\! \int_{t}^{T}\!\! \Phi(s, t)^\top  \big[ A^\top\! P_1 \!+\!P_1A \!+\! C^\top P_1 C\!+\! Q_1 \!-\! \bar{\Theta}_1^\top \big(R_2\!+\! D_1^\top P_1 D_1 \big)\bar{\Theta}_1 \big]  \Phi(s, t) ds.\vspace{-1mm}
\end{eqnarray}

{\bf Step (ii).}  Define $\Pi(s,t)=S(s)^{-1/2}\Phi(s,t)$. Then $\Phi(s,t)=S(s)^{1/2}\Pi(s,t)$, so boundedness of $\Phi$ is equivalent to boundedness of $\Pi$. Differentiating  $\Pi(s, t)$ with respect to $s$, we obtain 
\vspace{-1mm}
\begin{align*}
\frac{\partial \Pi (s,t)}{\partial s}  = \dot{T}(s) \Phi(s, t) + T(s) \frac{\partial \Phi(s,t)}{\partial s}  = \big[ \dot{T}(s)T^{-1}(s) - S^{1/2}(s) P_2(s) S^{1/2}(s) \big] \Pi(s, t),\vspace{-1mm}
\end{align*}
where $T(s)=S(s)^{-1/2}$.
From the variation of constants formula,  
\vspace{-1mm}
\begin{equation}\label{pr-th4-eq9}
\begin{aligned}
\Pi(s, t) &= \Psi(s, t) \Pi(t, t) + \int_{t}^{s} \Psi(s, \tau) \dot{T}(\tau)T^{-1}(\tau) \Pi(\tau, t) d\tau\\
&= \Psi(s, t) T(t) + \int_{t}^{s} \Psi(s, \tau) \dot{T}(\tau)T^{-1}(\tau) \Pi(\tau, t) d\tau
\end{aligned}\vspace{-1mm}
\end{equation}
where $\Psi$ is the fundamental solution of $\partial_s\Psi(s,t) = -S(s)^{1/2}P_2(s)S(s)^{1/2}\Psi(s,t)$, $\Psi(t,t)=I$. Let $V(s,t)=\Psi(s,t)^\top\Psi(s,t)$. Then 
$$\partial_s\operatorname{Tr}(V) = -2\operatorname{Tr}(\Psi^\top S^{1/2}P_2S^{1/2}\Psi) \leq 0,$$
which implies
\begin{equation*}
\begin{aligned}
\mathrm{Tr} (V(s,t))\leq \mathrm{Tr} (V(t,t))=\mathrm{Tr} (I)=n, \q \forall 0\leq t\leq s \leq T,
\end{aligned}
\end{equation*}
and
\begin{equation*}
|\Psi(s,t)|\leq \sqrt{n},\q \forall 0\leq t\leq s \leq T.
\end{equation*}
Combining this with \eqref{pr-th4-eq9},  the boundedness of $T(\cd),T(\cd)^{-1},\dot{T}(\cd)$, and Gr\"onwall's inequality, we obtain
\vspace{-1mm}
\begin{equation}\label{pr-th4-eq10}
\max\{|\Pi(s,t)|,|\Phi(s,t)|\}\leq \cC,\q \forall 0\leq t\leq s \leq T.
\end{equation}

\textbf{Step (iii).} In this step, we show the boundedness of $P_1$ and $P_2$.

Since $-\bar{\Theta}_1^\top(R_2+D_1^\top P_1D_1)\bar{\Theta}_1\leq0$, from the expression for $P_1(t)$ and the boundedness of $\Phi$ we obtain 
\begin{equation*}
\begin{aligned}
|P_1(t)|&\leq \cC|G_1| +\cC \int_{t}^{T} \big[ |P_1|(|A|+|C|^2) + |Q_1|\big] ds \leq \cC + \int_{t}^{T} \big( \cC|P_1| + \cC \big) ds,
\end{aligned}%\vspace{-1mm}
\end{equation*}
which implies the boundedness of $P_1$ by Gr\"onwall's inequality. Similar to Step 3, we can also obtain the boundedness of $P_2$.
\end{proof}

\end{document}